\documentclass{amsart}
\usepackage[utf8x]{inputenc}
\usepackage{amsmath, amsthm, amssymb}
\usepackage[usenames,dvipsnames,svgnames,table]{xcolor}
\usepackage[margin=1.3in]{geometry}
\usepackage{cleveref}

\newtheorem{theorem}{Theorem}[section]
\newtheorem{proposition}[theorem]{Proposition}

\newtheorem{lemma}[theorem]{Lemma}
\newtheorem{definition}[theorem]{Definition}
\newtheorem*{remark}{Remark}

\newcommand{\R}{\mathbb R}

\newcommand{\eps}{\varepsilon}

\newcommand{\dd}{\, \mathrm{d}}

\newcommand{\spt}{\mbox{spt}}
\newcommand{\tr}{\mbox{tr}}

\Crefname{assumption}{Assumption}{Assumptions}
\Crefname{theorem}{Theorem}{Theorems}
\Crefname{lem}{Lemma}{Lemmas}
\Crefname{cor}{Corollary}{Corollaries}
\Crefname{prop}{Proposition}{Propositions}
\Crefname{theorem}{Theorem}{Theorems}
\Crefname{conjecture}{Conjecture}{Conjectures}

\numberwithin{equation}{section}

\title[$C^\infty$ smoothing for the Landau equation]{$C^\infty$ smoothing for weak solutions of the inhomogeneous Landau equation}
\author{Christopher Henderson}
\address{Department of Mathematics, University of Chicago, 5734 S. University Ave., Chicago, IL 60637}
\curraddr{Department of Mathematics, University of Arizona, Tucson, AZ 85721}
\email{ckhenderson@math.arizona.edu}
\author{Stanley Snelson}
\address{Department of Mathematics, University of Chicago, 5734 S. University Ave., Chicago, IL 60637}
\curraddr{Department of Mathematical Sciences, Florida Insitute of Technology, Melbourne, FL 32901}
\email{ssnelson@fit.edu}

\thanks{Both authors were partially supported by National Science Foundation grant DMS-1246999.  CH was partially supported by NSF grant DMS-1907853. SS was partially supported by a Ralph E. Powe Award from ORAU}

\begin{document}

\begin{abstract}
We consider the spatially inhomogeneous Landau equation with initial data that is bounded by a Gaussian in the velocity variable. In the case of moderately soft potentials, we show that weak solutions immediately become smooth, and remain smooth as long as the mass, energy, and entropy densities remain under control. For very soft potentials, we obtain the same conclusion with the additional assumption that a sufficiently high moment of the solution in the velocity variable remains bounded. Our proof relies on the iteration of local Schauder-type estimates.
\end{abstract}

\maketitle

\section{Introduction}

The Landau equation from plasma physics models the evolution of a particle density $f(t,x,v)\geq 0$ in phase space, see e.g. \cite{chapmancowling, lifschitzpitaevskii}. In spatial dimension $d$, the equation is given by
\begin{align}
\partial_t f + v\cdot \nabla_x f &= Q_L(f,f) := \nabla_v \cdot \left( \int_{\R^d} a(v-w)[f(w)\nabla f(v) - f(v)\nabla f(w)] \dd w \right),\label{e:main}\\
a(z) &= a_{d,\gamma}|z|^{\gamma+2}\left(I - \frac z {|z|}\otimes \frac z{|z|}\right).
\end{align}
Here, $t\in [0,T_0]$, $x\in\R^d$, $v\in\R^d$, $\gamma \geq -d$, and $a_{d,\gamma}>0$ is a physical constant. The Landau equation arises as the limit of the Boltzmann equation as grazing collisions predominate \cite{alexandre2004landau}. We are interested in both the case of \emph{moderately soft potentials}, $\gamma\in (-2,0)$ and \emph{very soft potentials}, $\gamma \in [-d,-2]$. The case $d=3, \gamma = -3$ corresponds to Coulomb interaction between particles at small scales.

As opposed to the Boltzmann collision operator, which is a purely integro-differential operator of fractional order, $Q_L$ is an operator of diffusion type whose coefficients depend nonlocally on $f$. In particular, the Landau equation \eqref{e:main} can be written in divergence form
\begin{equation}\label{e:divergence}
\partial_t f + v\cdot \nabla_x f = \nabla_v\cdot \left[\overline a(t,x,v)\nabla_v f\right] + \overline b(t,x,v)\cdot\nabla_v f + \overline c(t,x,v) f,
\end{equation}
or in nondivergence form
\begin{equation}\label{e:nondivergence}
\partial_t f + v\cdot \nabla_x f = \tr\left[\overline a(t,x,v)D_v^2 f\right] + \overline c(t,x,v) f,
\end{equation}
with the coefficients $\overline a(t,x,v)\in \R^{d\times d}$, $\overline b(t,x,v) \in \R^d$, and $\overline c(t,x,v)\in \R$ defined by
\begin{align}
\overline a(t,x,v) &:= a_{d,\gamma}\int_{\R^d} \left( I - \frac w {|w|} \otimes \frac w {|w|}\right) |w|^{\gamma + 2} f(t,x,v-w) \dd w,\label{e:a}\\
\overline b(t,x,v) &:= b_{d,\gamma}\int_{\R^d} |w|^\gamma w f(t,x,v-w)\dd w,\label{e:b}\\
\overline c(t,x,v) &:= c_{d,\gamma}\int_{\R^d} |w|^\gamma f(t,x,v-w)\dd w, \label{e:c}
\end{align}
for some constants $a_{d,\gamma}$, $b_{d,\gamma}$, and $c_{d,\gamma}$. When $\gamma = -d$, the expression for $\overline c$ must be replaced by $c_{d,\gamma} f$. We use both formulations \eqref{e:divergence} and \eqref{e:nondivergence}, which are equivalent as long as, say, $f\in C^2_{v,\rm loc}$ and $f$ has enough decay so that $\overline a$, $\overline b$, and $\overline c$ are well-defined.

 We make the following assumptions on the mass density, energy density, and entropy density:
\begin{align}
0<m_0\leq \int_{\R^d} f(t,x,v)\dd v \leq M_0,\label{e:M0}\\
\int_{\R^d} |v|^2 f(t,x,v)\dd v \leq E_0,
	\qquad \text{and}\label{e:E0}\\
\int_{\R^d} f(t,x,v) \log f(t,x,v) \dd v \leq H_0,\label{e:H0}
\end{align}
uniformly in $t\geq 0$ and $x\in\R^d$. In the spatially homogeneous case, i.e.~when $f$ is assumed to be independent of $x$, the mass and energy are conserved, and the entropy is monotonically decreasing; hence, in this case, it would suffice to assume that the initial data have finite mass, energy, and entropy. It is not currently known whether these hydrodynamic quantities stay under control for $t>0$ in the inhomogeneous case, so we include \eqref{e:M0}, \eqref{e:E0}, and \eqref{e:H0} as \emph{a priori} assumptions.

We are interested in the regularity of weak solutions to \eqref{e:main}. We use the following notion of weak solution, which is implicitly used in \cite{golse2016} and \cite{cameron2017landau}:
\begin{definition}
We say $f:[0,T_0]\times \R^d\times\R^d\to\R_+$ is a weak solution of \eqref{e:divergence} if $f$, $\nabla_v f$, $\partial_t f + v\cdot \nabla_x f \in L^2_{loc}(\R^{2d+1})$, the coefficients $\overline a$, $\overline b$, and $\overline c$ are well-defined, and 
\[\int_{\R^{2d+1}} (\partial_t f + v\cdot \nabla_x f)\phi \dd v\dd x \dd t = \int_{\R^{2d+1}} \left(- \langle \overline a \nabla_v f, \nabla_v \phi\rangle + (\overline b \cdot \nabla_v f + \overline c f)\phi\right)\dd v \dd x \dd t \]
for all $\phi \in H_0^1(\R^{2d+1})$. 
\end{definition}

Our main result states that weak solutions immediately become smooth, for any initial data that is bounded by a Gaussian and regular enough for a weak solution to exist:


\begin{theorem}\label{t:main}
Let $\gamma \in (-2,0)$, and let $f:[0,T_0]\times\R^d\times\R^d\to\R_+$ be a bounded weak solution of the Landau equation \eqref{e:main} satisfying the bounds \eqref{e:M0}, \eqref{e:E0}, and \eqref{e:H0}. There exists $\mu_0>0$ depending on $d$, $\gamma$, $m_0$, $M_0$, $E_0$, and $H_0$ such that if the initial data $f_{\rm in}$ satisfies 
\[f_{\rm in}(x,v) \leq C_0 e^{-\mu|v|^2}\] 
for some $C_0>0$ and $\mu>0$, then $f\in C^\infty((0,T_0]\times \R^d\times\R^d)$, and for any $\mu'<\min\{\mu_0,\mu\}$, any integer $j \geq 0$, and any multi-indices $\beta$ and $\eta$ with non-negative integer coordinates, the partial derivatives of $f$ satisfy the pointwise estimates
\begin{equation}\label{e:pointwise}
|\partial_t^j\partial_x^\beta\partial_v^\eta f(t,x,v)| \leq C \left(1+t^{-q}\right)e^{-\mu'|v|^2}.
\end{equation}
The constants $C, q\geq 0$ depend on $d$, $\gamma$, $m_0$, $M_0$, $E_0$, $H_0$, $\mu'$, $j$, $|\beta|$, $|\eta|$, and $C_0$. 

For $\gamma\in [-d,-2]$, if we make the additional assumption that  for all $t\in[0,T_0]$ and $x\in \R^d$,
\begin{equation}\label{e:4thmoment}
\int_{\R^d}|v|^p f(t,x,v)\dd v \leq P_0,
\end{equation}
where $p$ is the smallest integer such that $p>\dfrac{d|\gamma|}{2+\gamma+d}$, then the same conclusion holds, with all constants depending additionally on $P_0$ and $\|f\|_{L^\infty([0,T_0]\times \R^d\times \R^d)}$. If $\gamma \leq -d/2-1$, the constants also depend on $T_0$. 
\end{theorem}
%


The question of global-in-time existence of smooth solutions to \eqref{e:main} for non-perturbative initial data remains a challenging open problem. In the case of moderately soft potentials, Theorem \ref{t:main} implies a physically meaningful continuation criterion: any loss of smoothness of $f$ can be detected at the macroscopic level by a breakdown of the bounds on the mass, energy, or entropy density.  

Our proof of Theorem \ref{t:main} relies on three elements:
\begin{enumerate}
\item[1.] The local H\"older continuity of solutions to \eqref{e:main}, which was established in \cite{wang2011ultraparabolic} and \cite{golse2016}.

\item[2.] Decay of the solution $f$ for large velocities, and corresponding decay in the local estimates, which is needed to pass regularity of $f$ to regularity of the coefficients $\overline a$ and $\overline c$ in \eqref{e:nondivergence}.

\item[3.] Local Schauder-type estimates for kinetic Fokker-Planck equations with H\"older continuous coefficients, which we prove in Section \ref{s:schauder} and apply iteratively in Section \ref{s:landau}.
\end{enumerate}
The second point is where our assumption that $f_{\rm in}$ is bounded by a Gaussian comes in. In \cite{cameron2017landau}, it was shown that this upper bound is propagated for all $t\in (0,T_0]$ when $\gamma\in (-2,0)$. We extend this to $\gamma\in [-d,-2]$ in Theorem \ref{t:gaussian}, under more restrictive assumptions; however, if we could guarantee by any other method that sufficiently high moments of the solution are finite (as in the hypotheses of \cite{chen2009smoothing} and \cite{liu2014regularization}, see below), our proof would still go through. It was shown in \cite{cameron2017landau} that solutions of \eqref{e:main} satisfying the hydrodynamic bounds \eqref{e:M0}, \eqref{e:E0}, and \eqref{e:H0} satisfy \emph{a priori} pointwise decay proportional to $(1+|v|)^{-1}$ for arbitrary initial data, but this is not strong enough for our purposes because of the slowly decaying kernels in \eqref{e:a} and \eqref{e:c}. It was also shown in \cite{cameron2017landau} that \emph{a priori} Gaussian decay cannot hold without any decay assumption on $f_{\rm in}(x,v)$.

\subsection{Related work}

In \cite{chen2009smoothing}, the authors show that classical solutions of \eqref{e:main} defined on a three-dimensional torus are $C^\infty$ in all three variables, provided that infinitely many moments of the solution and its first eight derivatives in $x$ and $v$ remain bounded uniformly in time and provided that the solution remains bounded away from vacuum. 
A corresponding result for solutions defined on $\R^3$ was shown in \cite{liu2014regularization}, in the case $\gamma\in [-3,-2)$. Our Theorem \ref{t:main} extends these results in the case where $f_{\rm in}$ is bounded by a Gaussian. 
The assumptions \eqref{e:M0}, \eqref{e:E0}, and \eqref{e:H0} are much weaker than the \emph{a priori} regularity hypotheses of \cite{chen2009smoothing} and \cite{liu2014regularization}, and are defined in terms of physically relevant hydrodynamic quanitites. At least in the case $\gamma\in (-2,0)$, our estimates do not depend quantitatively on the $L^\infty$ norm of $f$. 

Local H\"older estimates for kinetic equations with rough coefficients were proven by Wang-Zhang \cite{wang2011ultraparabolic} and Golse-Imbert-Mouhot-Vasseur \cite{golse2016}, and this is the starting point for the application of our Schauder estimates. The first global regularity estimates for \eqref{e:main} in this setting (weak solutions with bounded mass, energy, and entropy) were established in \cite{cameron2017landau}. The ellipticity constants of the diffusion operator $Q_L$ degenerate as $|v|\to \infty$ in a non-isotropic way (see Appendix \ref{s:A}). To deal with this, we use a change of variables derived in \cite{cameron2017landau} to obtain an equation with universal ellipticity constants in a small cylinder (see Lemma \ref{l:T}).

Regarding the existence theory for \eqref{e:main}, global-in-time classical solutions have only been constructed in the close-to-equilibrium setting: see the work of Guo \cite{guo2002periodic} in the $x$-periodic case, and Mouhot-Neumann \cite{mouhot2006equilibrium} in the whole space. For general initial data, Villani \cite{villani1996global} constructed so-called renormalized solutions with defect measure for the Landau equation. More recently, He-Yang \cite{he2014boltzmannlandau} established the short-time existence of spatially periodic classical solutions to \eqref{e:main} in the Coulomb case ($\gamma = -d$) with initial data in a weighted $H_{x,v}^7$ space, by taking the grazing collisions limit in their estimates on the Boltzmann collision operator. They assume that the mass density of the initial data is uniformly bounded away from zero. Since this lower bound along with the bounds \eqref{e:M0}, \eqref{e:E0}, \eqref{e:H0}, and \eqref{e:4thmoment} can be shown to propagate for a short time, our Theorem \ref{t:main} combined with \cite{he2014boltzmannlandau} provides a $C^\infty$ solution to the Cauchy problem for suitable initial data. However, on physical grounds, the equation should be expected to be well-posed even with vacuum regions in the initial data. We explore this issue, as well as short-time existence for a broader range of $\gamma$, in a forthcoming paper.



For the spatially homogeneous Landau equation, $C^\infty$ smoothing was established in \cite{desvillettes2000landau} in the case $\gamma>0$ and \cite{villani1998landau} in the $\gamma = 0$ case. For $\gamma \in (-2,0)$, the upper bounds of \cite{silvestre2015landau} also imply smoothing via parabolic regularity theory. For $\gamma\in [-d,-2]$, the result of Theorem \ref{t:main} is new even in the space homogeneous case, to the best of our knowledge.


\subsection{Schauder estimates}

Our main technical tools are local Schauder-type estimates for linear kinetic Fokker-Planck equations of the form
\begin{equation}\label{e:holder}
\partial_t u + v\cdot \nabla_x u = \tr(AD_v^2 u) + g,
\end{equation}
with $A$ and $g$ H\"older continuous (see Theorem \ref{t:weak-schauder} below). Schauder estimates have been established in the more general setting of ultraparabolic equations by Manfredini \cite{manfredini1997ultraparabolic},   DiFrancesco-Polidoro \cite{difrancesco2006schauder}, and Bramanti-Brandolini \cite{bramanti2007schauder}, among others. However, there are two complications involved in bootstrapping regularity estimates in this context: based on the natural scaling of the equation, Schauder estimates should be expected to bound two derivatives in $v$, one derivative in $t$, and two-thirds of a derivative in $x$ (i.e. the $\frac 2 3$-H\"older norm in $x$) of $u$, which is not enough to directly conclude $u$ is a classical solution. Even worse, Schauder estimates do not provide $C^\alpha$ estimates on $\partial_t u$, but rather on $\partial_t u + v\cdot\nabla_x u$. This is related to the non-symmetric Lie group structure of the equation, which shows up in the representation formula \eqref{e:convolution} of the solution. 
To get around this, we prove a second estimate that bounds $\partial_t u$ and $\nabla_x u$ in terms of the $C^{1+\alpha}$-norm of $g$. We give elementary proofs of the estimates we need, using the explicit fundamental solution for constant-coefficient equations.

\subsection{Organization of the paper} In Section \ref{s:schauder}, we prove regularity estimates for kinetic equations with H\"older continuous coefficients. In Section \ref{s:landau}, we apply these estimates iteratively to weak solutions of the Landau equation.  
In Appendix \ref{s:A}, we review the bounds on the coefficients $\overline a$, $\overline b$, and $\overline c$ in \eqref{e:nondivergence}.

\subsection{Notation} 

We let $z=(t,x,v)$ denote a point in $\R_+\times \R^d\times \R^d$. For any $z_0=(t_0,x_0,v_0)$, define the Galilean transformation
\[\mathcal S_{z_0}(t,x,v) := (t_0+t, x_0 + x +tv_0,v_0+v).\]
We also have
\[\mathcal S_{z_0}^{-1}(t,x,v) := (t-t_0, x - x_0 -(t-t_0)v_0,v-v_0).\]
For $r>0$, define the scaling $\delta_r$ by 
\[\delta_r(t,x,v) = (r^2t,r^3x,rv).\]
The class of equations of the form \eqref{e:holder} is invariant under $S_{z_0}$ and $\delta_r$. We also define the quasimetric
\[\rho(z,z') := \|\mathcal S_{z}^{-1} z'\| = |t'-t|^{1/2} + |x'-x - (t'-t)v|^{1/3} + |v'-v|,\]
where
\[\|z-z'\| := |t-t'|^{1/2} + |x-x'|^{1/3} + |v-v'|.\]
For any $r>0$ and $z_0 = (t_0,x_0,v_0)$, let
\[ Q_r(z_0) := 
	(t_0-r^2,t_0] \times \{x : |x-x_0 - (t-t_0) v_0| < r^3 \}\times B_r(v_0),  \]
and $Q_r = Q_r(0,0,0)$.

We say a constant is universal if it depends only on $\gamma$, $d$, $m_0$, $M_0$, $E_0$, and $H_0$ when $\gamma \in (-2,0)$. When $\gamma\in [-d,-2]$, we also allow universal constants to depend on $P_0$ and $\|f\|_{L^\infty([0,T_0]\times \R^d\times \R^d)}$. The notation $A\lesssim B$ means that $A\leq CB$ for a constant $C$ that depends on the quantities listed in the statement of the given lemma or theorem, and $A\approx B$ means that $A\lesssim B$ and $B\lesssim A$.  


\section{Schauder estimates for linear kinetic equations}\label{s:schauder}

In this section, we obtain regularity estimates for equations of the form \eqref{e:holder}. We begin by defining H\"older norms and semi-norms that correspond to $\rho$.
\begin{definition} Let $Q\subseteq \R^{2d+1}$. For $u:Q\to \R$, define
\begin{align*}
[u]_{\alpha,Q} &:= \sup_{\substack{z,z'\in Q,\\ z\neq z'}} \frac{|u(z) - u(z')|}{\rho(z,z')^\alpha}\\
[u]_{\alpha,x,Q} &:= \sup_{\substack{(t,x,v),(t,x',v)\in Q,\\ x\neq x'}} \frac{|u(t,x,v) - u(t,x',v)|}{|x-x'|^{\alpha}},\\
[u]_{\alpha,t,Q} &:= \sup_{\substack{(t,x,v),(t',x,v)\in Q,\\ t\neq t'}}\frac{|u(t,x,v) - u(t',x,v)|}{|t-t'|^{\alpha}+|(t'-t)v|^{2\alpha/3}}\\
|u|_{0,Q} &:= \sup_{z\in Q} |u(z)|\\
|u|_{\alpha,Q} &:= |u|_{0,Q} + [u]_{\alpha,Q}\\
[u]_{1+\alpha,Q} &:= [\nabla_v u]_{\alpha,Q_1} + [u]_{(1+\alpha)/2,t,Q} + [u]_{(1+\alpha)/3,x,Q}\\
|u|_{1+\alpha,Q} &:= |u|_{0,Q} + |\nabla_v u|_{0,Q} + [u]_{1+\alpha,Q}\\
[u]_{2+\alpha,Q} &:= [D_v^2 u]_{\alpha,Q} + \left[\partial_t u\right]_{\alpha,Q} + [u]_{(2+\alpha)/3,x,Q}\\
|u|_{2+\alpha,Q} &:= |u|_{0,Q} + |\partial_t u|_{0,Q} + |\nabla_v u|_{0,Q} + |D_v^2 u|_{0,Q} + [u]_{2+\alpha,Q}.
\end{align*}
For $\beta\in (0,3)$, if $|u|_{\beta,Q} <\infty$, we say $u\in C^{\beta}(Q)$. 
\end{definition}
If $u$ is in $C^\alpha(Q)$ by this definition, then in particular, $u$ is $\frac \alpha 3$-H\"older continuous in the Euclidean metric on $\R^{2d+1}$. We use the following lemma repeatedly:
\begin{lemma}[Interpolation Inequalities]\label{l:interp}
Let $Q = Q_r(z_0)$ for some $z_0\in \R^{2d+1}$ and $r>0$, and let $u\in C^{2+\alpha}(Q)$. There exists a constant $C$, depending only on the dimension, such that for any $\eps>0$,
\begin{align*}
[u]_{\alpha,Q} &\leq \eps^{2}[u]_{2+\alpha,Q} + C\eps^{-\alpha}|u|_{0,Q},\\
|\partial_t u|_{0,Q} &\leq \eps^\alpha [\partial_t u]_{\alpha,Q} + C\eps^{-2}|u|_{0,Q},\\
|\nabla_v u|_{0,Q} &\leq \eps^{1+\alpha}[u]_{2+\alpha,Q} + C\eps^{-1}|u|_{0,Q},\\
[\nabla_v u]_{\alpha,Q} &\leq \eps [u]_{2+\alpha,Q} + C\eps^{-(1+\alpha)}|u|_{0,Q},\\
|D_v^2 u|_{0,Q} &\leq \eps^\alpha [u]_{2+\alpha,Q} + C\eps^{-2}|u|_{0,Q}.
\end{align*}
If $D_v^3 u, \nabla_x u\in C^{\alpha}(Q)$, we also have
\begin{align*}
|D_v^3 u|_{0,Q} &\leq \eps^\alpha[D_v^3 u]_{\alpha,Q} + C\eps^{-3}|u|_{0,Q}\\
|\nabla_x u|_{0,Q} &\leq \eps^{\alpha} [\nabla_x u]_{\alpha,Q} + C\eps^{-3}|u|_{0,Q}.
\end{align*}
\end{lemma}
The method of proving inequalities of this type is standard. (See, for example, \cite{manfredini1997ultraparabolic} or \cite[Theorem~8.8.1]{krylov_holder}). Briefly, it suffices to prove the case $\eps = 1$ by scaling. To prove the first inequality, one estimates $|u(z) - u(z')|$ by writing $z-z'$ as a sum of segments parallel to the coordinate axes, and applying the mean value inequality. The details are omitted. 

Finally, we define the non-scale-invariant H\"older seminorms that correspond to our regularity estimates:

\begin{definition}\label{d:prime}
For $Q \subseteq \R^{2d+1}$, $u:Q\to \R$, and $\alpha,\beta\in (0,1)$, define
\begin{align*}
 [u]_{2+\alpha,\beta,Q}' :&= [D_v^2 u]_{\alpha,Q} + [u]_{(2+\alpha)/3,x,Q} + [u]_{\beta,t,Q},\\
[u]_{3+\alpha,Q}'' &:= [\partial_t u]_{\alpha,Q} + [\nabla_x u]_{\alpha,Q} + [D_v^3 u]_{\alpha,Q}.
 \end{align*}
\end{definition}

\subsection{Constant coefficients} 
Consider the equation
\begin{equation}\label{e:simple}
u_t + v\cdot \nabla_x u - \Delta_v u = g,
\end{equation}
in $(-1,0]\times \R^d\times \R^d $ 
with zero initial data at $t=-1$. The explicit fundamental solution for this equation is given by
\begin{equation}\label{e:Gamma}
\Gamma(z) := \begin{cases}\dfrac {C_d}{t^{2d}} \exp\left(-\dfrac{|v|^2}{t} - \dfrac {3 v\cdot x}{t^2} - \dfrac {3|x|^2}{t^3}\right), &t>0,\\
0, &t\leq 0,\end{cases}
\end{equation}
where $C_d = (\sqrt{3} / (2\pi))^d$. More precisely, if $g$ is, say, continuous, bounded, and has support contained in $\{t > -1\}$ then \eqref{e:simple} is uniquely solved by
\begin{equation}\label{e:convolution}
u(z) = \int_{\R^{2d+1}} \Gamma\left(\mathcal S_{\zeta}^{-1} z\right) g(\zeta) \dd \zeta,
\end{equation}
where $\zeta = (s,y,w)$ and $\mathcal S_{\zeta}^{-1} z = (t-s,x-y-(t-s)w,v-w)$. The fundamental solution $\Gamma$ is a special case of the solution constructed by H\"ormander \cite{hormander1967} for more general hypoelliptic equations. (See also \cite{lanconelli1994evolution, manfredini1997ultraparabolic}.) The following lemma provides a useful characterization of the homogeneity of the fundamental solution:

\begin{lemma}\label{l:Gamma-estimates}
For any partial derivative $\partial_t^j \partial^\beta_x\partial^\eta_v \Gamma$ of $\Gamma$, with $\beta$ a multi-index of order $k\geq 0$ and $\eta$ a multi-index of order $\ell\geq 0$, there exists a constant $C = C(d,j,k,\ell)$ such that for all $t>0$ and $p,q\geq 0$,
\[\int_{\R^d}\int_{\R^d}|\partial_t^j \partial^\beta_x\partial^\eta_v \Gamma(t+\xi_1,y+\xi_2,w+\xi_3)||y|^p |w|^q \dd w \dd y \leq C t^{-(\ell/2 + j + 3k/2)+3p/2 + q/2}.  \]
Further, if $\xi \in [0,1]\times \R^d \times \R^d$ and $\|\xi\| \leq t^{1/2}/2$, then
\[
	\int_{\R^d}\int_{\R^d}|\partial_t^j\partial^\beta_x\partial^\eta_v \Gamma(z + \xi)||y|^p |w|^q \dd w \dd y
		\leq C t^{-(\ell/2 + j + 3k/2) + 3p/2 + q/2},
\]
where $z = (t,x,v)$.
\end{lemma}
\begin{proof}
It is straightforward to show by induction that every partial derivative of $\Gamma$ can be written 
\[\partial_t^j \partial^\beta_x\partial^\eta_v \Gamma(t,y,w) = P_{j,\beta,\eta}\left(\frac 1 {t^{1/2}} , \frac {y_1} {t^2}, \ldots,\frac {y_d}{t^2},\frac {w_1} {t}, \ldots,\frac {w_d}{t} \right)\Gamma(t,y,w),\]
with $P_{j,\beta,\eta}$ a homogeneous polynomial where each term is of degree exactly $\ell + 2j + 3k$. Since $\exp( -|w|^2/t - 3 w\cdot y / t^2 - 3 |y|^2 / t^3) \leq \exp(- |w|^2/(16 t) - 3|y|^2/(5t^3))$, formula \eqref{e:Gamma} for $\Gamma$ implies
\begin{align*}
\int_{\R^d}\int_{\R^d}  |\partial_t^j \partial^\beta_x\partial^\eta_v&\Gamma(t,y,w)||y|^p |w|^q  \dd w \dd y\\
& =  \frac {C_d} {t^{2d}} \int_{\R^d} \int_{\R^d} \left| P_{j,\beta,\eta}\left(\frac 1 {t^{1/2}} , \frac {y_1} {t^2}, \ldots,\frac {y_d}{t^2},\frac {w_1} {t}, \ldots,\frac {w_d}{t} \right)\Gamma(t,y,w)\right||y|^p |w|^q\dd w \dd y\\
&\leq C_d  \int_{\R^d} \int_{\R^d} \left|P_{j,\beta,\eta}\left(\frac 1 {t^{1/2}}, \frac {y_1} {t^{1/2}}, \ldots,\frac {y_d}{t^{1/2}},\frac {w_1} {t^{1/2}}, \ldots,\frac {w_d}{t^{1/2}} \right)\right.\\
&\quad \qquad \qquad\qquad\left. \times \exp\left( -\dfrac{|\overline w|^2}{16} - \dfrac {3|\overline y|^2}{5 }\right)\right| t^{3p/2+q/2}|\overline y|^p |\overline w|^q \dd \overline w \dd \overline y\\
&\lesssim  \left(\frac 1 {t^{1/2}}\right)^{\ell + 2j + 3k} t^{3p/2+q/2},
\end{align*}
where $\overline w = w/t^{1/2}$ and $\overline y = y/t^{3/2}$.
The proof of the second claim is almost identical, using the fact that $t\lesssim t + \xi_t \lesssim t$, where $\xi := (\xi_t, \xi_x, \xi_v)$.
\end{proof}


We now prove our main regularity estimates in the constant-coefficient case:
\begin{lemma}\label{l:convolution}
Suppose that $g \in C^\alpha(Q_1)$ has compact support in $Q_1$, for some $\alpha\in(0,1)$. Then the solution $u$ of \eqref{e:simple} in $Q_1$ satisfies
\[\begin{split}
	[D_v^2u]_{\alpha,Q_1} + [u]_{(2+\alpha)/3,x,Q_1}
		&\lesssim [g]_{\alpha,Q_1},
\end{split}\]
where the implied constant depends only on $\alpha$ and the dimension $d$. We also have $[u]_{\beta,t,Q_1}\lesssim [g]_{\alpha,Q_1}$ for any $\beta\in (0,1)$, so that 
\[[u]_{2+\alpha,\beta,Q_1}' \lesssim [g]_{\alpha,Q_1}, \]
with $[\cdot]_{2+\alpha,\beta,Q_1}$ as in Definition \ref{d:prime}. In particular, $[u]_{1+\beta,Q_1} \lesssim [g]_{\alpha,Q_1}$
for any $\beta\in (0,1)$.
\end{lemma}
\begin{proof}
First, we estimate $[D_v^2 u]_{\alpha,Q_1}$. Since $g$ has compact support in $Q_1$, \eqref{e:convolution} implies that, for any $(t,x,v) \in Q_1$,
\begin{align*}
\partial_{v_iv_j} u(z) &= \int_{-1}^t \int_{\R^d}\int_{\R^d} \partial_{v_iv_j}\Gamma(t-s,x-y-(t-s)w,v-w) g(s,y,w)\dd w \dd y \dd s\\
&=  \int_{0}^{1+t} \int_{\R^d}\int_{\R^d}\partial_{v_iv_j}\Gamma(s,y,w) g(t-s,x-y-s(v-w),v-w)\dd w\dd y \dd s,
 \end{align*}
for $1\leq i,j\leq d$. Let $z = (t,x,v)$ and  $z' = (t',x',v')$ be fixed points in $Q_1$ with $t \leq t'$.  Further, let $h =\rho(z,z')$ and fix any $i,j\in \{1,\dots, d\}$. We write
\begin{align*}
&\partial_{v_iv_j} u(z) - \partial_{v_iv_j}u(z')\\
&= \left(\int_0^{2h^2} + \int_{2h^2}^{1+t}\right)\int_{\R^{d}}\int_{\R^d}\partial_{v_iv_j}\Gamma(s,y,w) \delta g(s,y,w) \dd w\dd y \dd s\\
&\quad - \int_{1+t}^{1+t'} \int_{\R^{d}}\int_{\R^d}\partial_{v_iv_j}\Gamma(s,y,w)g(t'-s,x'-y-s(v'-w),v'-w)\dd w\dd y \dd s\\
&=: I_1 + I_2 + I_3,
\end{align*}
where
\[\delta g(s,y,w) := g(t-s,x-y - s(v-w),v-w) - g(t'-s,x'-y-s(v'-w),v'-w).\] 
We make the convention that if $2h^2 \geq 1+t$, then $I_2=0$. 

Since $\spt(g) \subset Q_1$, we have $|\delta g(s,y,w)-\delta g(s,y,0)| \leq 2[g]_{\alpha,Q_1}((s|w|)^{\alpha/3} + |w|^{\alpha})$. Observe that for any $s>0$, $y\in \R^d$,
\[\int_{\R^d} \partial_{v_iv_j}\Gamma(s,y,w) \dd w = 0.\]
This allows us to estimate $I_1$ as follows:
\begin{align*}
|I_1| &= \left|\int_0^{2h^2}\int_{\R^{d}}\int_{\R^d} \partial_{v_iv_j}\Gamma(s,y,w) [\delta g(s,y,w)-\delta g(s,y,0)]\dd w \dd y \dd s\right|\\
&\leq 2 [g]_{\alpha,Q_1} \int_0^{2h^2} \int_{\R^{d}}\int_{\R^d}|\partial_{v_iv_j}\Gamma(s,y,w)| ((s|w|)^{\alpha/3} + |w|^{\alpha})\dd w \dd y \dd s\\
&\lesssim  [g]_{\alpha,Q_1}\int_0^{2h^2} s^{\alpha/2 - 1} \dd s \lesssim [g]_{\alpha,Q_1}h^{\alpha},
\end{align*}
where the second-to-last inequality follows from Lemma \ref{l:Gamma-estimates}. 

Changing variables in $I_2$ and adding and subtracting a term, we have
\begin{equation*}
\begin{split}
I_2
	&= \int_{-1}^{t-2h^2} \int_{\R^{d}}\int_{\R^d} [\partial_{v_iv_j}\Gamma(t - s,x-y,v-w)g(s,y-(t-s)w,w)\\
	&\qquad \qquad \qquad - \partial_{v_iv_j} \Gamma(t'-s,x'-y,v'-w)g(s,y-(t'-s)w,w)]\dd w \dd y \dd s\\
	&= \int_{-1}^{t-2h^2} \int_{\R^{d}}\int_{\R^d} \partial_{v_iv_j}\Gamma(t - s,x-y,v-w) \\
	&\qquad \qquad \qquad \times [g(s,y-(t-s)w,w) - g(s,y-(t'-s)w,w)]\dd w \dd y \dd s\\
	&\quad + \int_{-1}^{t-2h^2} \int_{\R^{d}}\int_{\R^d}[\partial_{v_iv_j}\Gamma(t - s,x-y,v-w) - \partial_{v_iv_j}\Gamma(t' - s,x'-y,v'-w)]\\
	&\qquad \qquad \qquad \times g(s,y-(t'-s)w,w) \dd w \dd y \dd s\\
	&=: I_2' + I_2''. 
\end{split}
\end{equation*}
Re-defining $\delta g(s,y,w) := g(s,y-(t-s)w,w) - g(s,y-(t'-s)w,w)$, we have
\[|\delta g(s,y,w) - \delta g(s,y,v)| \leq [g]_{\alpha,Q_1}\left((|t-s|^{1/3}+|t'-s|^{1/3})|v-w|^{1/3} + 2|v-w|\right),\]
which implies
\begin{align*}
|I_2'| &= \left|\int_{-1}^{t-2h^2}  \int_{\R^d}\int_{\R^d} \partial_{v_iv_j}\Gamma(t-s,x-y,v-w)[\delta g(s,y,w) - \delta g(s,y,v)] \dd w \dd y \dd s\right|\\
&\lesssim [g]_{\alpha,Q_1}\int_{2h^2}^{1+t} \int_{\R^d}\int_{\R^d}  |\partial_{v_iv_j}\Gamma(s,y,w)| \left((s^{\alpha/3} + |t'-t+s|^{\alpha/3})|w|^{\alpha/3} + |w|^\alpha\right)\dd w \dd y \dd s\\
&\lesssim [g]_{\alpha,Q_1} \int_{2h^2}^{1+t} s^{-1} \left(s^{\alpha/2} + h^{2\alpha/3} s^{\alpha/6}\right)\dd s
\lesssim [g]_{\alpha,Q_1} h^{\alpha},
\end{align*}
by Lemma \ref{l:Gamma-estimates}. For $I_2''$, first note that
\[\begin{split}
	I_2''
		= \int_{-1}^{t-2h^2} \int_{\R^d}\int_{\R^d} &[\partial_{v_iv_j} \Gamma(t-s,x-y,v-w) - \Gamma(t'-s,x'-y,v'-w)]\\
		& \times [g(s,y-(t'-s)w,w) - g(s,y-(t'-s)v,v)]\dd w \dd y \dd s.
\end{split}\]
We next note that, with $\zeta = (s,y,w)$,
\begin{align*}
|\partial_{v_iv_j}&\Gamma(t - s,x-y,v-w) - \partial_{v_iv_j}\Gamma(t' - s,x'-y,v'-w)|\\
 &\leq \max_{\|\xi\|\leq h, \xi_1 \geq 0} \left(h^2|\partial_t\partial_{v_iv_j}\Gamma(z-\zeta+\xi)| + h^3|\nabla_x\partial_{v_iv_j}\Gamma(z-\zeta+\xi)|+h|\nabla_v\partial_{v_iv_j}\Gamma(z-\zeta+\xi)|\right),
\end{align*}
where we denote $\xi = (\xi_1, \xi_2, \xi_3)\in \R\times \R^d\times \R^d$.

Using these two facts along with the second half of Lemma~\ref{l:Gamma-estimates}, we have
\begin{align*}
|I_2''| 
&\lesssim  [g]_{\alpha,Q_1}\int^{1+t}_{2h^2} \int_{\R^d}\int_{\R^d} \max_{\|\xi\|\leq h, \xi_1\geq 0}\big[ h^2|\partial_t\partial_{v_iv_j} \Gamma(\zeta + \xi)| + h^3|\nabla_x\partial_{v_iv_j}\Gamma(\zeta + \xi)|\\
 &\qquad +h|\nabla_v\partial_{v_iv_j}\Gamma(\zeta + \xi)|\big]
  \left(|t'-t+s+\xi_1|^{\alpha/3}|w-\xi_3|^{\alpha/3} + |w-\xi_3|^\alpha\right)\dd w \dd y \dd s\\
 &\lesssim  [g]_{\alpha,Q_1}\int^{1+t}_{2h^2} \int_{\R^d}\int_{\R^d} \max_{\|\xi\|\leq h, \xi_1\geq 0}\big[ h^2|\partial_t\partial_{v_iv_j} \Gamma(\zeta + \xi)| + h^3|\nabla_x\partial_{v_iv_j}\Gamma(\zeta + \xi)|\\
 &\qquad +h|\nabla_v\partial_{v_iv_j}\Gamma(\zeta + \xi)|\big]
  \left((h^2 +s)^{\alpha/3}(|w|+ h)^{\alpha/3} + |w|^\alpha + h^\alpha\right)\dd w \dd y \dd s\\
  &\lesssim [g]_{\alpha,Q_1} h^\alpha.
\end{align*}


Proceeding as in our estimate of $I_1$, with $g(t'-s,x'-y-s(v'-w),v'-w)$ playing the role of $\delta g(s,y,w)$, we obtain
\begin{align*}
|I_3|\lesssim  [g]_{\alpha,Q_1} \left((1+t')^{\alpha/2} - (1+t)^{\alpha/2}\right) \lesssim [g]_{\alpha,Q_1} |t'-t|^{\alpha/2} \lesssim [g]_{\alpha,Q_1} h^\alpha,
\end{align*}
completing the estimate of $[D_v^2 u]_{\alpha,Q_1}$.

To estimate the $C^{(2+\alpha)/3}$ norm of $u$ in the $x$ variable, we define $h = |x'-x|$ and write
\begin{align*}
&u(t,x',v) - u(t,x,v) \\
&= \left(\int_0^{h^{2/3}} + \int_{h^{2/3}}^{1+t}\right) \int_{\R^d}\int_{\R^d} \Gamma(s,y,w)\\
&\qquad  \times[g(t-s,x'-y-s(v-w),v-w) - g(t-s,x-y-s(v-w),v-w)]\dd w \dd y \dd s\\
&=: J_1 + J_2.
\end{align*}
Since $\int_{\R^d}\int_{\R^d} \Gamma(s,y,w)\dd w \dd y = 1$ for any $s>0$, we have
\begin{align*}
|J_1| &\leq [g]_{\alpha,Q_1} h^{\alpha/3} \int_0^{h^{2/3}} \int_{\R^d}\int_{\R^d} \Gamma(s,y,w) \dd w \dd y \dd s\\
&\leq [g]_{\alpha,Q_1} h^{(2+\alpha)/3}.
\end{align*}
For $J_2$, we use a change of variables and then the fact that
\[
	\int_{\R^d}\int_{\R^d} \Gamma(s, x' - y, w) \dd y \dd w
		= \int_{\R^d}\int_{\R^d} \Gamma(s, x - y, w) \dd y \dd w
\]
to rewrite the convolution as follows:
\begin{align*}
|J_2| &= \left|\int_{h^{2/3}}^{1+t} \int_{\R^d}\int_{\R^d} [\Gamma(s,x'-y,w) - \Gamma(s,x-y,w)] g(t-s,y-s(v-w),v-w) \dd y \dd w \dd s\right|\\
&= \left|\int_{h^{2/3}}^{1+t} \int_{\R^d}\int_{\R^d}[\Gamma(s,x'-y,w) - \Gamma(s,x-y,w)]\right.\\
&\qquad \qquad \times [g(t-s,y-s(v-w),v-w) - g(t-s,x-s(v-w),v-w)]\dd y \dd w \dd s\Big|\\
&\leq  [g]_{\alpha,Q_1} h\int_{h^{2/3}}^{1+t} \int_{\R^d}\int_{\R^d}  \left( \max_{|\xi|\leq h}|\nabla_x \Gamma(s,x-y+\xi,w)| \right)|x-y|^{\alpha/3} \dd y \dd w \dd s\\
&\lesssim [g]_{\alpha,Q_1} h  \max_{|\xi|\leq h}\int_{h^{2/3}}^{1+t} \left(s^{-3/2 + \alpha/2} + s^{-3/2}|\xi|^{\alpha/3}\right)\dd s
\lesssim [g]_{\alpha,Q_1} h^{(2+\alpha)/3},
\end{align*}
using Lemma \ref{l:Gamma-estimates}, that $|\xi|\leq h$, and that $h \leq s^{3/2}$ on the domain of integration.

The proof that $[u]_{\beta,t,Q_1} \leq [g]_{\alpha,Q_1}$ follows a similar outline, and is omitted.
\end{proof}

\begin{lemma}\label{l:dt-dx}
With $g$ and $u$ as in Lemma \textup{\ref{l:convolution}}, assume in addition that $g\in C^{1+\alpha}(Q_1)$ for some $\alpha \in (0,1)$. Then $u$ satisfies
\[[u]_{3+\alpha,Q_1}'' =  [\partial_t u]_{\alpha,Q_1} + [\nabla_x u]_{\alpha,Q_1} + [D_v^3 u]_{\alpha,Q_1} \leq C[g]_{1+\alpha,Q_1},\]
where the constant depends on $\alpha$ and $d$.
\end{lemma}
\begin{proof}
First, we show the estimate $[\nabla_x u]_{\alpha,Q_1} \leq C[g]_{1+\alpha,Q_1}$. We proceed as in the previous lemma, taking advantage of the regularity of $g$ in $x$. We have
 \begin{align*}
\partial_{x_i} u(z) 
&=  \int_{0}^{1+t} \int_{\R^d}\int_{\R^d}\partial_{x_i}\Gamma(s,y,w) g(t-s,x-y-s(v-w),v-w)\dd w\dd y \dd s,
 \end{align*}
for $1\leq i\leq d$. Let $z,z' \in Q_1$ with $t \leq t'$, and 
let $h = \rho(z,z')$. We write
\begin{align*}
&\partial_{x_i} u(z) - \partial_{x_i}u(z')\\
&= \left(\int_0^{2h^2} + \int_{2h^2}^{1+t}\right)\int_{\R^{d}}\int_{\R^d}\partial_{x_i}\Gamma(s,y,w) \delta g(s,y,w) \dd w\dd y \dd s\\
&\quad - \int_{1+t}^{1+t'} \int_{\R^{d}}\int_{\R^d}\partial_{x_i}\Gamma(s,y,w)g(t'-s,x'-y-s(v'-w),v'-w)\dd w\dd y \dd s\\
&=: I_1 + I_2 + I_3,
\end{align*}
where
\[\delta g(s,y,w) := g(t-s,x-y - s(v-w),v-w) - g(t'-s,x'-y-s(v'-w),v'-w).\] 
We make the convention that if $2h^2 \geq 1+t$, then $I_2=0$. 

Since $\spt(g) \subset Q_1$, we have $|\delta g(s,y,w)-\delta g(s,0,w)| \leq 2[g]_{\alpha,Q_1}|y|^{(1+\alpha)/3}$. Observe that for any $s>0$, $y\in \R^d$,
\[\int_{\R^d} \partial_{x_i}\Gamma(s,y,w) \dd y = 0.\]
This allows us to estimate $I_1$ as follows:
\begin{align*}
|I_1| &= \left|\int_0^{2h^2}\int_{\R^{d}}\int_{\R^d} \partial_{x_i}\Gamma(s,y,w) [\delta g(s,y,w)-\delta g(s,0,w)]\dd y \dd w \dd s\right|\\
&\leq 2 [g]_{\alpha,Q_1} \int_0^{2h^2} \int_{\R^{d}}\int_{\R^d}|\partial_{x_i}\Gamma(s,y,w)| |y|^{(1+\alpha)/3}\dd y \dd w \dd s\\
&\lesssim [g]_{\alpha,Q_1}\int_0^{2h^2} s^{-3/2 + (1+\alpha)/2} \dd s \lesssim [g]_{\alpha,Q_1}h^{\alpha},
\end{align*}
by Lemma \ref{l:Gamma-estimates}. 

Changing variables in $I_2$, we have
\begin{equation*}
\begin{split}
I_2
	&= \int_{-1}^{t-2h^2} \int_{\R^{d}}\int_{\R^d} [\partial_{x_i}\Gamma(t - s,x-y,v-w)g(s,y-(t-s)w,w)\\
	&\qquad \qquad \qquad - \partial_{x_i} \Gamma(t'-s,x'-y,v'-w)g(s,y-(t'-s)w,w)]\dd y \dd w \dd s\\
	&= \int_{-1}^{t-2h^2} \int_{\R^{d}}\int_{\R^d} \partial_{x_i}\Gamma(t - s,x-y,v-w) \\
	&\qquad \qquad \qquad \times [g(s,y-(t-s)w,w) - g(s,y-(t'-s)w,w)]\dd y \dd w \dd s\\
	&\quad + \int_{-1}^{t-2h^2} \int_{\R^{d}}\int_{\R^d}[\partial_{x_i}\Gamma(t - s,x-y,v-w) - \partial_{x_i}\Gamma(t' - s,x'-y,v'-w)]\\
	&\qquad \qquad \qquad \times g(s,y-(t'-s)w,w) \dd y \dd w \dd s\\
	&=: I_2' + I_2''. 
\end{split}
\end{equation*}
Re-defining $\delta g(s,y,w) := g(s,y-(t-s)w,w) - g(s,y-(t'-s)w,w)$, we have
\[|\delta g(s,y,w) - \delta g(s,x,w)| \leq 2[g]_{1+\alpha,Q_1}|x-y|^{(1+\alpha)/3},\]
which implies
\begin{align*}
|I_2'| &= \left|\int_{-1}^{t-2h^2}  \int_{\R^d}\int_{\R^d} \partial_{x_i}\Gamma(t-s,x-y,v-w)[\delta g(s,y,w) - \delta g(s,x,w)] \dd y \dd w \dd s\right|\\
&\lesssim [g]_{\alpha,Q_1}\int_{2h^2}^{1+t} \int_{\R^d}\int_{\R^d}  |\partial_{x_i}\Gamma(s,y,w)| |y|^{(1+\alpha)/3}\dd y \dd w \dd s\\
&\lesssim [g]_{\alpha,Q_1} \int_{2h^2}^{1+t} s^{-3/2+(1+\alpha)/2} \dd s \lesssim [g]_{\alpha,Q_1} h^{\alpha},
\end{align*}
by Lemma \ref{l:Gamma-estimates}. For $I_2''$, first note that with $\zeta = (s,y,w)$,
\begin{align*}
|\partial_{x_i}&\Gamma(t - s,x-y,v-w) - \partial_{x_i}\Gamma(t' - s,x'-y,v'-w)|\\
 &\leq \max_{\|\xi\|\leq h} \left(h^2|\partial_t\partial_{x_i}\Gamma(z-\zeta+\xi)| + h^3|\nabla_x\partial_{x_i}\Gamma(z-\zeta+\xi)|+h|\nabla_v\partial_{x_i}\Gamma(z-\zeta+\xi)|\right).
\end{align*}
By applying Lemma \ref{l:Gamma-estimates} again and arguing as in the proof of Lemma \ref{l:convolution}, we have
\begin{align*}
|I_2''| &= \left|\int_{-1}^{t-2h^2} \int_{\R^{d}}\int_{\R^d}[\partial_{x_i}\Gamma(t - s,x-y,v-w) - \partial_{x_i}\Gamma(t' - s,x'-y,v'-w)]\right.\\
	&\qquad \times [g(s,y-(t'-s)w,w) - g(s,x-(t'-s)w,w)] \dd y \dd w \dd s\Big|\\
 &\lesssim [g]_{1+\alpha,Q_1}\int_{2h^2}^{1+t} \int_{\R^d}\int_{\R^d} \max_{\|\xi\|\leq h}[ h^2|\partial_t\partial_{x_i} \Gamma(s,y,w)| + h^3|\nabla_x\partial_{x_i}\Gamma(s,y,w)|\\
 &\qquad +h|\nabla_v\partial_{x_i}\Gamma(s,y,w)|]
|y-\xi_2|^{(1+\alpha)/3}\dd y \dd w \dd s\\
 &\lesssim [g]_{1+\alpha,Q_1} h^\alpha.
\end{align*}

Proceeding as in our estimate of $I_1$, with $g(t'-s,x'-y-s(v'-w),v'-w)$ playing the role of $\delta g(s,y,w)$, we obtain
\begin{align*}
|I_3|\leq C  [g]_{1+\alpha,Q_1} \left((1+t')^{\alpha/2} - (1+t)^{\alpha/2}\right) \leq C [g]_{1+\alpha,Q_1} |t'-t|^{\alpha/2} \leq C[g]_{1+\alpha,Q_1} h^{\alpha},
\end{align*}
and the proof of the estimate on $[\nabla_x u]_{\alpha,Q_1}$ is complete.

Equation \eqref{e:simple} and Lemma \ref{l:convolution} imply the estimate on $[\partial_t u]_{\alpha,Q_1}$. We complete the proof by differentiating \eqref{e:simple} in $v$ and applying Lemma \ref{l:convolution} to estimate $[D_v^3 u]_{\alpha,Q_1}$, using our already-established estimate on $\nabla_x u$.
\end{proof}

Next, let $A_0$ be a (constant) symmetric, strictly positive definite, $d\times d$ matrix. Assume that $\sigma(A_0)\subset [\lambda,\Lambda]$ where $0<\lambda < \Lambda$.
\begin{lemma}\label{l:constant-coeffs}
If $g\in C^\alpha(Q_1)$ for some $\alpha\in (0,1)$, and $g$ 
has compact support in $Q_1$, then the solution $u$ of 
\textup{\[\partial_t u + v\cdot \nabla_x u - \tr(A_0 D_v^2 u) = g\]}
satisfies
\begin{align*}
[u]_{2+\alpha,\beta,Q_1}' &\leq C [g]_{\alpha,Q_1},
\end{align*}
for any $\beta\in (0,1)$. If, in addition, $g\in C^{1+\alpha}(Q_1)$, then
\begin{align*}
[u]_{3+\alpha,Q_1}'' &\leq C [g]_{1+\alpha,Q_1}.
\end{align*}
The constants $C$ depend on $d$, $\alpha$, $\beta$, $\lambda$, and $\Lambda$.
\end{lemma}
\begin{proof}
Let $P$ be such that $P^2 = A_0$, and define $u_P(t,x,v) := u(t,Px,Pv)$. Notice that $\sigma(P) \subset [\sqrt \lambda, \sqrt \Lambda]$.  Then 
\[\partial_t u_P + v\cdot \nabla_x u_P - \Delta_v u_P = (\partial_t u + v\cdot\nabla_x u-\Delta_v u)(t,Px,Pv) = g(t,Px,Pv) =: g_P(t,x,v),\] 
and we can apply Lemma \ref{l:convolution} to $u_P = \displaystyle\int \Gamma (\mathcal S_{\zeta}^{-1} z) g_P(\zeta)\dd \zeta$ to obtain 
\begin{align*}
[u]_{2+\alpha,P(Q_{1})} &\leq C(P) [g]_{\alpha,P(Q_1)},
\end{align*}
where $P(Q_1) := (-1,0]\times P(B_1)\times P(B_1)$. To get an estimate on $Q_1$, we replace $u$ with $u(R^2t, R^3x,Rv)$, where $R>0$ depends only on $\lambda$ and $\Lambda$.  
Similarly, if $D_v^3u, \nabla_x u,\partial_t u\in C^{\alpha}(Q_1)$, we apply Lemma \ref{l:dt-dx} to $u_P$.
\end{proof}

\subsection{Variable coefficients}

Let $L$ be an operator of the form
\[Lu = \tr(A(z)D_v^2 u),\]
where $A\in C^{\alpha}(Q_1)$, and $0<\lambda I \leq A(z) \leq \Lambda I$ for all $z\in Q_1$.  We now study equations of the form
\begin{equation}\label{e:variable}
\partial_t u + v\cdot \nabla_x u - Lu = g.
\end{equation}
As is standard, we  extend Lemma \ref{l:constant-coeffs} to solutions of \eqref{e:variable} by freezing the coefficients at a point $z$ and taking advantage of the closeness of $L$ to $L(z)$ in a small cylinder around $z$, where $L(z)$ refers to the operator $\tr(A(z) D_v^2 u)$ with $z$ ``frozen''. We also remove the assumption that $u$ has compact support, which requires tracking how interior estimates on $Q_r$ scale for $r\in (0,1]$. For this, we need the following technical lemma: 
\begin{lemma}\label{l:omega}
Let $\omega(r)\geq 0$ be bounded in $[r_0,r_1]$ with $r_0\geq 0$. Suppose for $r_0\leq r<R\leq r_1$, we have
\[\omega(r) \leq \mu \omega(R) + \frac A {(R-r)^p} + B\]
for some $\mu \in [0,1)$ and $A,B,p \geq 0$. Then for any $r_0\leq r<R\leq r_1$, there holds
\[\omega(r) \lesssim\left(\frac A {(R-r)^p} + B\right),\]
where the implied constant depends only on $\mu$ and $p$.
\end{lemma}
\begin{proof}
See \cite[Lemma 4.3]{hanlin}.
\end{proof}

\begin{theorem}\label{t:Schauder}
Fix $\alpha \in (0,1)$.  Suppose that, $[u]_{2+\alpha,\beta,Q_1}' < \infty$ for all $\beta \in (0,1)$, and $A\in C^{\alpha}(Q_1)$. Then 
\[ [u]_{2+\alpha,\beta,Q_{1/2}}' \lesssim \left([g]_{\alpha,Q_1} + |A|_{\alpha,Q_1}^{3+\alpha+2/\alpha} |u|_{0,Q_1}\right),\]
where \textup{$g := \partial_t u + v\cdot \nabla_x u - Lu$}. The implied constant depends only on $d,\alpha, \beta, \lambda$, and $\Lambda$.
\end{theorem}
\begin{proof}
For $r\in (0,1]$, recall that
\[ [u]_{2+\alpha,\beta,Q_r}' = [D_v^2 u]_{\alpha,Q_r} + [u]_{(2+\alpha)/3,x,Q_r} + [u]_{\beta,t,Q_r}.\]
Let $r\in [\frac 1 4, \frac 3 4]$ be arbitrary. 
 For $1\leq i,j\leq d$, pick $z,z'\in Q_{r}$ such that 
\[\frac{|\partial_{v_iv_j} u(z) - \partial_{v_iv_j} u(z')|}{\rho(z,z')^\alpha} \geq \frac 1 2 [\partial_{v_iv_j} u]_{\alpha,Q_{r}}.\]
Let $\theta\in (0,1/8)$ be a constant, to be chosen later. If $\rho(z,z') \geq \theta$, then by the interpolation inequalities in Lemma~\ref{l:interp}, 
\begin{equation}\label{eq:chris3}
	[\partial_{v_iv_j} u]_{\alpha,Q_{r}}
		\leq 2\theta^{-\alpha} |D_v^2 u|_{0,Q_{r}} \leq \frac{1}{12d^2} [u]_{2+\alpha,\beta,Q_{r}}' + C\theta^{-2}|u|_{0,Q_{r}}.
\end{equation}
On the other hand, if $\rho(z,z') < \theta$, let $\chi$ be a smooth cutoff such that $\chi(\tilde z) = 1$ if $\rho(\tilde z,z') < \theta$ and $\chi(\tilde z) = 0$ if $\rho(\tilde z,z') \geq 2\theta$. We can choose $\chi$ such that 
\[\begin{split}
	&|\nabla_v \chi|_{0,Q_1} \lesssim \theta^{-1},
	\quad [\nabla_v \chi]_{0,Q_1} \lesssim \theta^{-1-\alpha},
	\quad |\partial_t \chi+v\cdot \nabla_x \chi|_{0,Q_1} + |D_v^2 \chi|_{0,Q_1} \lesssim \theta^{-2},\\
	&\text{ and } 
	\quad [\partial_t \chi+v\cdot \nabla_x \chi]_{\alpha,Q_1} + [D_v^2 \chi]_{\alpha,Q_1} \lesssim \theta^{-2-\alpha}.
\end{split}\]
Using Lemma \ref{l:constant-coeffs}, we now have
\begin{align*}
[\partial_{v_iv_j} u]_{\alpha,Q_{r}} &\leq 2[\chi u]_{2+\alpha,\beta,Q_{r+2\theta}}'\\
&\lesssim [\partial_t (\chi u) + v\cdot \nabla_x (\chi u) - L(z')(\chi u)]_{\alpha,Q_{r+2\theta}}\\
&\lesssim [\partial_t(\chi u) + v\cdot \nabla_x (\chi u) - L(\chi u)]_{\alpha,Q_{r+2\theta}} + [(L - L(z'))(\chi u)]_{\alpha,Q_{r+2\theta}}.
\end{align*}
Let $R = r+2\theta$. To estimate the first term on the last line, note that
\[\partial_t(\chi u) + v\cdot \nabla_x (\chi u) - L(\chi u) = \chi g + u(\partial_t + v\cdot \nabla_x -L)\chi - 2(A(z)\nabla_v u)\cdot \nabla_v \chi.\]
By the interpolation inequalities in Lemma \ref{l:interp}, 
\begin{equation}\label{eq:chris1}
\begin{split}
[\partial_t(\chi u) +& v\cdot \nabla_x (\chi u) - L(\chi u)]_{\alpha,Q_{R}} \lesssim ([g]_{\alpha,Q_{R}} + (1+|A|_{0,Q_1})(\theta^{-2}[u]_{\alpha,Q_{R}} + \theta^{-1}[\nabla_v u]_{\alpha,Q_{R}}))\\
&\lesssim [g]_{\alpha,Q_{R}} + (1+|A|_{0,Q_1})\left(\theta^{\alpha} [u]_{2+\alpha,\beta,Q_{R}}' +  C\theta^{-2-\alpha(2+\alpha)}|u|_{0,Q_{R}}\right).
\end{split}
\end{equation}
For the second term, note that $(L-L(z'))(\chi u) = \tr((A(\tilde z)-A(z'))D_v^2(\chi u))$ for all $\tilde z \in Q_1$. Since $\spt(\chi) \subset \{\tilde z : \rho(\tilde z,z')\leq 2\theta\}$, we have
\begin{equation}\label{eq:chris2}
\begin{split}
[(L-L(z'))(\chi u)]_{\alpha,Q_R} &\lesssim [A]_{\alpha,Q_1} \theta^\alpha \left([D^2_v u]_{\alpha,Q_R} + |D_v^2 u|_{0,Q_R}\right)\\
&\lesssim [A]_{\alpha,Q_1} \theta^\alpha \left( [u]_{2+\alpha,\beta,Q_R}' + \theta^{-2}|u|_{0,Q_R}\right),
\end{split}
\end{equation}
using the interpolation inequalities again. Combining~\eqref{eq:chris1} and~\eqref{eq:chris2}, we obtain, when $\rho(z,z') <\theta$, 
\begin{equation}\label{eq:chris4}
	[\partial_{v_iv_j} u]_{\alpha,Q_r}
		\lesssim |A|_{\alpha,Q_1}\theta^\alpha \left([u]_{2+\alpha,\beta,Q_R}' + [g]_{\alpha,Q_R} + \theta^{-p}|A|_{\alpha,Q_1}|u|_{0,Q_1}\right),
\end{equation}
with $p = 2+\alpha(2+\alpha)$.

The combination of~\eqref{eq:chris3} and~\eqref{eq:chris4} implies that, for any fixed $\theta\in(0,1/8)$,
\[[\partial_{v_iv_j} u]_{\alpha,Q_r} \leq \left(C|A|_{\alpha,Q_1}\theta^\alpha+ \frac{1}{12d^2}\right) [u]_{2+\alpha,\beta,Q_R}' + C[g]_{\alpha,Q_R} + C\theta^{-p}|A|_{\alpha,Q_1}|u|_{0,Q_1}.\]
Summing over $i$ and $j$, and applying a similar argument to $[u]_{(2+\alpha)/3,x,Q_r}$ and $[u]_{\beta,t,Q_r}$, we obtain
\[[u]_{2+\alpha,\beta,Q_r}' \leq \left(C|A|_{\alpha,Q_1}\theta^\alpha+ \frac 1 4\right) [u]_{2+\alpha,\beta,Q_R}' + C[g]_{\alpha,Q_R} + C\theta^{-p}|A|_{\alpha,Q_1}|u|_{0,Q_1}.\]
Fix $\theta_0>0$ such that $C|A|_{\alpha,Q_1}\theta^\alpha < 1/4$ for all $\theta\in (0,\theta_0)$. Then, for each $R\in (r,r+2\theta_0)$, we have
\[[u]_{2+\alpha,\beta,Q_r}' \leq \frac 1 2 [u]_{2+\alpha,\beta,Q_R}' + C[g]_{Q_R} + C(R-r)^{-p}|A|_{\alpha,Q_1}|u|_{0,Q_1}.\]

Recall that $r\in [\frac 1 4, \frac 3 4]$ was arbitrary. Lemma \ref{l:omega} with $\omega(s) = [u]_{2+\alpha,\beta,Q_s}'$, $r_0 = 1/2$, and $r_1 = 1/2+2\theta_0$ implies
\[[u]_{2+\alpha,\beta,Q_r}' \leq  C([g]_{\alpha,Q_1} + (R-r)^{-p}|A|_{\alpha,Q_1}|u|_{0,Q_1}),\]
for each $\frac 1 2 \leq r < R \leq \frac 1 2 + 2\theta_0$. Choose $r=\frac 1 2$ and $R = \frac 1 2 + \theta_0$, and the proof is complete.
\end{proof}

Next, we extend the estimate of Lemma \ref{l:dt-dx} to the variable-coefficient case. Here, we need to assume $A(z)$ in the operator $L$ is in $C^{1+\alpha}(Q_1)$.

\begin{theorem}\label{t:higher}
Assume that $D_v^3 u, \nabla_x u, \partial_t u\in C^{\alpha}(Q_1)$. Then 
\[ [u]_{3+\alpha,Q_{1/2}}''\leq C \left(|g|_{1+\alpha,Q_1} + |A|_{1+\alpha,Q_1}^{5+\alpha+6/\alpha}|u|_{0,Q_1}\right),\]
where \textup{$g := \partial_t u + v\cdot \nabla_x u - \tr(AD_v^2u)$}. The constant $C$ depends on $d, \alpha, \lambda$, and $\Lambda$.
\end{theorem}

\begin{proof}
For $r\in (0,1]$, recall
\[[u]_{3+\alpha,Q_r}'' = [\partial_t u]_{\alpha,Q_r} + [\nabla_x u]_{\alpha,Q_r} + [D_v^3 u]_{\alpha,Q_r}.\]
With $r$, $\theta$, and $R$ as in the proof of Theorem \ref{t:Schauder}, we can follow the argument of that proof to show
\[ [u]_{3+\alpha,Q_{r}}'' \leq \left( C|A|_{1+\alpha,Q_1} \theta^\alpha + \frac 1 4\right)[u]_{3+\alpha,Q_R}'' + C|g|_{1+\alpha,Q_R} + C\theta^{\alpha(4+\alpha)+6}|A|_{1+\alpha,Q_1}|u|_{0,Q_1}.\]
The conclusion of the proof is the same as Theorem \ref{t:Schauder}.
\end{proof}


In the previous theorems, we assumed that solutions exist.  This is verified by the following theorem:
%
\begin{proposition}\label{p:exist}
Given $g \in C^{\alpha}((-1,0]\times \R^d\times\R^d)$ with compact support in $Q_1$, then there exists a unique weak solution $u$ in $C^{2+\alpha}|((-1,0]\times\R^d\times\R^d)$ of~\eqref{e:variable}.  Furthermore, $[u]_{2+\alpha,\beta,Q_1}'<\infty$. 
If $g\in C^{1+\alpha}(Q_1)$, the same conclusion holds with $[u]_{3+\alpha,Q_1}'' <\infty$, where $[\cdot]_{2+\alpha,\beta,Q_1}'$ and $[\cdot]_{3+\alpha,Q_1}''$ are as in Definition \textup{\ref{d:prime}}. 
\end{proposition}
\begin{proof}
Fix any $\beta \in (0,1)$ and assume that the matrix $A$ is uniformly bounded and coercive on $\R\times\R^d \times \R^d$. Define the norm
\[
	\|u\|_{\mathcal{B}} := \max \left\{ |u|_{\alpha, Q_1(z_0)}+[u]_{2+\alpha,\beta, Q_1(z_0)} + [\partial_tu + v\cdot \nabla_x u]_{\alpha, Q_1(z_0)} : z_0 = (0,x_0,v_0), x_0,v_0\in\R^d\right\},
\]
and the Banach space
\[
	\mathcal{B} := \{u \in C^{2+\alpha}([-1,0]\times \R^d\times \R^d) : \|\cdot\|_{\mathcal B} <\infty\},\]
endowed with $\|\cdot\|_{\mathcal B}$, 
and
\[
	\mathcal{V} := \{u \in C^\alpha([-1,0]\times \R^d \times \R^d) : u(-1,\cdot,\cdot) \equiv 0\},
\]
endowed with the obvious norm.

For any $\theta \in [0,1]$, define the operator $E_\theta: \mathcal{B} \to \mathcal{V}$ by
\[
	E_\theta u := u_t + v\cdot \nabla_x u - (1-\theta) \Delta_v u - \theta Lu.
\]
From \Cref{t:Schauder}, we see that
\[
	\|u\|_{\mathcal{B}} \lesssim \|E_\theta u\|_{\mathcal{V}}
\]
for all $u\in \mathcal{B}$.  Linearity and the above inequality imply that $E_\theta$ is injective.  Also, from~\eqref{e:Gamma}, 
we see that $E_0$ is onto.  Applying the method of continuity as in~\cite[Theorem~5.2]{gilbargtrudinger}, we obtain that $E_1$ is onto as well.  Hence, $E_1$ is invertible. 

This finishes the first claim.  The same argument applies in the second case when $g$ has one more derivative, using Theorem \ref{t:higher}.
\end{proof}

We collect all estimates above and use the equation that $u$ solves in order to derive an estimate on $(\partial_t + v\cdot\nabla_x)u$ to obtain the following theorem:
\begin{theorem}\label{t:weak-schauder}
Let $u$ be such that 
\[
	\partial_t u + v\cdot \nabla_x u - L u = g
\]
in $Q_1$, with \textup{$L = \tr(AD_v^2 u)$} and $\lambda I \leq A \leq \Lambda I$.
\begin{enumerate}
\item[\textup{(a)}] If $g, A\in C^{\alpha}(Q_1)$ 
for some $\alpha \in (0,1)$, we have the estimate
\[ [D_v^2u]_{\alpha,Q_{1/2}} + [u]_{(2+\alpha)/3,x,Q_{1/2}} + [u]_{\beta,t,Q_{1/2}}
	+[(\partial_t + v\cdot\nabla_x) u]_{\alpha,Q_{1/2}}\lesssim ( [g]_{\alpha,Q_1} + |A|_{\alpha,Q_1}^p|u|_{0,Q_1}),\]
for any $\beta\in (0,1)$. 
\item[\textup{(b)}] If $g, A\in C^{1+\alpha}(Q_1)$ for some $\alpha \in (0,1)$, then 
\[ [\partial_t u]_{\alpha,Q_{1/2}}+ [\nabla_x u]_{\alpha,Q_{1/2}} + [D_v^3 u]_{\alpha,Q_{1/2}} \lesssim (|g|_{1+\alpha,Q_1}  + |A|_{1+\alpha,Q_1}^q|u|_{0,Q_1}).\]
\end{enumerate}
The implied constants depend on $d$, $\alpha$, $\beta$, $\lambda$, and $\Lambda$. The exponents $p,q >0$ depend only on $\alpha$. 
\end{theorem}
%
%
%
%
%

\section{Smoothing for weak solutions of the Landau equation}\label{s:landau}

In this section, we apply the estimates of Section \ref{s:schauder} to the Landau equation. The diffusion operator $\tr(\overline a(z)D_v^2 f)$ (or in divergence form, $\nabla_v\cdot (\overline a(z)\nabla_v f)$) is uniformly elliptic in any bounded set, but the ellipticity constants degenerate as $|v|\to \infty$. (See Appendix \ref{s:A}.) To deal with this, we apply a change of variables in a small cylinder around a given point $z_0$, which yields an equation with ellipticity constants that are independent of $z_0$.  In the sequel, we undo this transformation to explicitly see the dependence of the estimates on $|v|$.

The following lemma was first proven in \cite{cameron2017landau} in the case of moderately soft potentials:
\begin{lemma}\label{l:T}
Let $z_0 =(t_0,x_0,v_0)\in \R_+\times \R^{d}\times \R^d$ be such that $|v_0|\geq 2$, and let $T$ be the linear transformation such that
\begin{equation*}
T e = \begin{cases}   |v_0|^{1+\gamma/2} e , & e \cdot v_0 = 0\\
|v_0|^{\gamma/2}e, & e \cdot v_0 = |v_0|.\end{cases}
\end{equation*}
Let $\tilde T(t,x,v) = (t,Tx,Tv)$, and define
\begin{align*}
\mathcal T_{z_0}(t,x,v) &:= \mathcal S_{z_0} \circ \tilde T (t,x,v)\\
& = (t_0+t,x_0+T x + t v_0 ,v_0 + T v).
\end{align*}
Then:
\begin{enumerate}
\item[\textup{(a)}] There exists a constant $C>0$ independent of $v_0\in\R^d\setminus B_2$ such that for all $v\in B_1$,
\[ C^{-1} |v_0| \leq |v_0 + Tv| \leq C |v_0|.\]
\item[\textup{(b)}] Let $f$ be a weak solution of the Landau equation \eqref{e:divergence} satisfying \eqref{e:M0}, \eqref{e:E0}, and \eqref{e:H0}, and if $\gamma < -2$, assume that $f$ satisfies \eqref{e:4thmoment}. Then there exists a radius 
\[r_1 = c_1 
		|v_0|^{-(1+\gamma/2)_+}\min\left(1,\sqrt{t_0/2}\right),\] with $c_1$ universal, such that for any $r\in (0,r_1]$, the function $f_{z_0}(t,x,v) := f(\mathcal T_{z_0}(r^2t,r^3x,rv))$ satisfies
\begin{equation}\label{e:isotropic}
\partial_t f_{z_0} + v \cdot \nabla_x f_{z_0} = \nabla_v \cdot\left(\overline A(z)\nabla_v f_{z_0}\right) +  \overline B(z)\cdot \nabla_v f_{z_0} +  \overline C(z) f_{z_0},
\end{equation}
or equivalently,
\begin{equation}\label{e:isotropic-nondivergence}
\partial_t f_{z_0} + v \cdot \nabla_x f_{z_0} = \textup{\tr}\left( \overline A(z)D_v^2 f_{z_0}\right) +  \overline C(z) f_{z_0},
\end{equation}
in $Q_1$, and the coefficients 
\[\begin{split}
	&\overline A(z) = T^{-1}\overline a(\mathcal T_{z_0}(\delta_r(z))) T^{-1},
	 \quad \overline B(z) = rT^{-1}\overline b(\mathcal T_{z_0}(\delta_r(z))),~\text{ and}\\ 
	 &\overline C(z) = r^2\overline c(\mathcal T_{z_0}(\delta_r(z)))
\end{split}\]
satisfy 
\begin{align*}
\lambda I &\leq \overline A(z) \leq \Lambda I,\\
	|\overline B(z)|
		&\lesssim \begin{cases} 1, &-1\leq \gamma<0 ,\\[2ex]
				|v_0|^{\min\{1+\gamma/2,0\}}\left(1+\|f(t,x,\cdot)\|_{L^\infty(B_\theta(v))}\right)^{-(\gamma+1)/d}, &-2\leq \gamma < -1,\\[2ex]
				|v_0|^{-\gamma/2 - 2 - 2(\gamma +1)/d}\left(1+\|f(t,x,\cdot)\|_{L^\infty(B_\theta(v))}\right)^{-(\gamma+1)/d}, &-d \leq \gamma <-2,
\end{cases}\\
	|\overline C(v)|
		&\lesssim \begin{cases} |v_0|^{-2}\left(1+\|f(t,x,\cdot)\|_{L^\infty(B_\theta(v))}\right)^{-\gamma/d}, &\dfrac{-2d}{d+2}\leq \gamma < 0,\\
|v_0|^{-(2+\gamma)_+ - 2 -2\gamma/d}\left(1+\|f(t,x,\cdot)\|_{L^\infty(B_\theta(v))}\right)^{-\gamma/d}, &-d < \gamma < \dfrac{-2d}{d+2},\end{cases}
\end{align*}
with $\lambda$ and $\Lambda$ universal, and $\theta \lesssim 1+ |v_0|^{-2/d}$. 
\end{enumerate}
\end{lemma}
\begin{proof}
For $\gamma\in (-2,0)$, this lemma is proven in \cite[Lemma 4.1]{cameron2017landau}. In fact, that proof does not use $\gamma>-2$ in an essential way. The necessary ingredients are the upper and lower bounds of Proposition \ref{p:a} and Lemma \ref{l:very-soft} from the Appendix, which hold under our assumptions on $f$. The bounds on $\overline B$ and $\overline C$ come from Proposition \ref{p:bc} and Lemma \ref{l:very-soft}.
\end{proof}

The coefficients $\overline A$, $\overline B$, and $\overline C$ are dependent on $z_0$, which we refer to as the ``base point,'' and~$r$.  

For any $z_0 = (t_0,x_0,v_0)$ with $|v_0|\leq 2$, we define $f_{z_0}(z) = f(\mathcal S_{z_0}\delta_{r_1} z)$, with $r_1$ as in Lemma \ref{l:T}(b). Note that in the notation of \cite{cameron2017landau}, our $f_{z_0}(t,x,v)$ is equal to $f_T(r_1^2t,r_1^3x,r_1v)$. The following proposition shows how the regularity of $f$ depends on the regularity of $f_{z_0}$.
\begin{proposition}\label{p:f0-f}
Let $f:[0,T_0]\times \R^d\times\R^d\to\R_+$ for some $T_0 > 0$. If $f_{z_0}$ is defined with base point $z_0\in (0,T_0]\times \R^d\times\R^d$, and some partial derivative $\partial_t^j \partial_x^\beta\partial_v^\eta f_{z_0}$ of order $M = 2j+3|\beta| +|\eta|$ exists in $C^\alpha(Q_1)$ for some $\alpha \in (0,1)$, then
\begin{align*}
|\partial_t^j \partial_x^\beta\partial_v^\eta f|_{\alpha,Q_{r_1}(z_0)} &\lesssim r_1^{-M-\alpha}(1+|v_0|)^{-\gamma \alpha/2} |\partial_t^j \partial_x^\beta\partial_v^\eta f_{z_0}|_{\alpha,Q_{1}}\\
&\lesssim \left(1+t_0^{-(M+\alpha)/2}\right) (1+|v_0|)^{M(1+\gamma/2)+\alpha}|\partial_t^j \partial_x^\beta\partial_v^\eta f_{z_0}|_{\alpha,Q_1}, 
\end{align*}
with $r_1$ as in Lemma \textup{\ref{l:T}}. 
\end{proposition}
\begin{proof}
%
Let $\partial = \partial_t^j\partial_x^\beta\partial_v^\eta$. For $z,z'\in Q_{r_1}(z_0)$ with $|v_0|\geq 2$, we have
\begin{align*}
|\partial f(z) - \partial f(z')| &= r_1^{-M}|\partial f_{z_0}(\delta_{r_1}^{-1}\mathcal S_{z_0}^{-1} \tilde T^{-1} z) - \partial f_{z_0}(\delta_{r_1}^{-1}\mathcal S_{z_0}^{-1} \tilde T^{-1} z')|\\
&\leq [\partial f_{z_0}]_{\alpha,Q_1} r_1^{-M}\rho(\delta_{r_1}^{-1}\mathcal S_{z_0}^{-1} \tilde T^{-1}z,\delta_{r_1}^{-1}\mathcal S_{z_0}^{-1} \tilde T^{-1}z')^{\alpha}\\
&= [\partial f_{z_0}]_{\alpha,Q_1} r_1^{-M-\alpha}\rho(\mathcal S_{z_0}^{-1} \tilde T^{-1}z,\mathcal S_{z_0}^{-1} \tilde T^{-1}z')^{\alpha}\\
&\leq [\partial f_{z_0}]_{\alpha,Q_1} r_1^{-M-\alpha}\rho(\tilde T^{-1}z,\tilde T^{-1}z')^\alpha\\
&\leq [\partial f_{z_0}]_{\alpha,Q_1} r_1^{-M-\alpha} |v_0|^{-\gamma \alpha/2} \rho(z,z')^\alpha.
\end{align*}
In the case $|v_0|\leq 2$, we have $f(z) = f_{z_0}(\delta_{r_1}^{-1} \mathcal S_{z_0}^{-1} z)$, and a similar calculation applies.
\end{proof}

Next, we show that if the regularity estimates of $f_{z_0}$ decay sufficiently quickly as $|v|\to\infty$, they imply regularity of the coefficients of \eqref{e:isotropic-nondivergence}. Although it is enough to show that partial derivatives of $\overline A$ and $\overline C$ grow at most polynomially, we derive explicit rates for the sake of concreteness.

\begin{lemma}\label{l:coefficients}
Let $f_{z_0}$ be as in Lemma \textup{\ref{l:T}}. Assume that some partial derivative $\partial_t^j\partial_x^\beta\partial_v^\eta f_{z_0}$ of order $M = j + |\beta| + |\eta|$ exists in $C^{\alpha}(Q_1)$ for every $z_0\in (0,T_0]\times \R^d\times\R^d$, and satisfies
\[[\partial_t^j \partial_x^\beta \partial_v^\eta f_{z_0}]_{\alpha,Q_1} \leq C_0 \left(1+t_0^{-p}\right) (1+|v_0|)^{-q}\]
for some $p\geq 0$ and $q > d+ 2 + \gamma(1-\alpha/2)+\alpha/3$. Then $\overline A(t,x,v)$ and $\overline C(t,x,v)$ enjoy the same regularity as $f_{z_0}$, and for any $z_0\in (0,T_0]\times \R^d\times\R^d$, one has 
\begin{align*}
\left[\partial_t^j\partial_x^\beta\partial_v^\eta \overline A\right]_{2\alpha/3,Q_1}  &\lesssim \left(1+t_0^{-M/2-p}\right)(1+|v_0|)^{(M+\alpha/3)(1+\gamma/2)_+ +2+\alpha +\alpha\gamma/3}\\
\left[\partial_t^j\partial_x^\beta\partial_v^\eta\overline C\right]_{2\alpha/3,Q_1} &\lesssim \left(1+t_0^{-M/2-p+1}\right)(1+|v_0|)^{(M+\alpha/3-2)(1+\gamma/2)_+ +\alpha+\alpha\gamma/3},
\end{align*}
where $\overline A$ and $\overline C$ are defined with base point $z_0$, and $r_1$ is as in Lemma \textup{\ref{l:T}}.  The implied constant depends on $d$, $\gamma$, $q$, and $C_0$.
\end{lemma}
\begin{proof}
Let $\partial= \partial_t^j \partial_x^\beta \partial_v^\eta$. For some base point $z_0$ with $|v_0|\geq 2$, fix $z,z'\in Q_1$ and let $\tilde z = (\tilde t,\tilde x,\tilde v) = \mathcal T_{z_0}(\delta_{r_1}z)$ and $\tilde z' = \mathcal T_{z_0}(\delta_{r_1}z')$, with $r_1$ as in Lemma \ref{l:T}. For $w\in \R^d$, Proposition \ref{p:f0-f} implies
\begin{align*}
	|\partial f(\tilde t,\tilde x,\tilde v-w) - &\partial f(\tilde t',\tilde x',\tilde v'-w)|
		\leq [\partial f]_{\alpha,Q_{r_1}(t_0,x_0,v_0-w)}\rho((\tilde t, \tilde x, \tilde v-w),(\tilde t', \tilde x', \tilde v'-w))^\alpha\\
		&\lesssim (1+t_0^{-p})r_1^{-M-\alpha}(1+|v_0-w|)^{-q-\gamma\alpha/2}(\rho(\tilde z,\tilde z')^\alpha+ |w|^{\alpha/3}\rho(\tilde z, \tilde z')^{2\alpha/3}).
\end{align*}
Recall $\overline A(z) = T^{-1}\overline a(\mathcal T_{z_0}(\delta_{r_1}z)) T^{-1}$. The formula \eqref{e:a} for $\overline a$ implies
\begin{align*}
|\partial \overline A(z) - \partial \overline A(z')| &\leq |v_0|^{-\gamma} \int_{\R^d} |w|^{\gamma+2}|\partial f(\tilde t,\tilde x,\tilde v-w) - \partial f(\tilde t',\tilde x',\tilde v'-w)|\dd w \\
 &\lesssim (1+t_0^{-p})|v_0|^{-\gamma} r_1^{-M-\alpha} \rho(\tilde z,\tilde z')^{2\alpha/3} \int_{\R^d} |w|^{\gamma+2+\alpha/3} (1+|v_0-w|)^{-q-\gamma\alpha/2}\dd w\\
&\lesssim (1+t_0^{-p})r_1^{-M-\alpha} \rho(\tilde z,\tilde z')^{2\alpha/3} |v_0|^{2+\alpha/3}\\
&\lesssim (1+t_0^{-p})r_1^{-M-\alpha/3} \rho(z,z')^{2\alpha/3} |v_0|^{2+\alpha+\alpha\gamma/3},
\end{align*}
where we have used $\rho(\tilde z,\tilde z') \lesssim |v_0|^{1+\gamma/2} r_1 \rho(z,z')$. A similar calculation applies to $\overline C(z) = r_1^2\overline c(\mathcal T_{z_0}\delta_{r_1}z)$. In the borderline case $\gamma = -d$, we have $\overline C(z) = c_{d,\gamma} r_1^2 f_{z_0}(z)$, and the conclusion of the lemma follows from the even stronger decay of $\partial f_{z_0}$.
\end{proof}
\begin{remark}
The decay in the estimates of Lemma \textup{\ref{l:coefficients}} can be improved when $|\eta|> 0$ by integrating by parts in $w$. However, this would still not grant us enough decay to conclude $f\in C^\infty$ without any decay assumption on the initial data.
\end{remark}

Next, we show that Gaussian bounds in the initial data are propagated. This result was established in the case $\gamma\in (-2,0)$ in \cite[Theorem 1.2]{cameron2017landau}
, under the assumption that the hydrodynamic bounds \eqref{e:M0}, \eqref{e:E0}, and \eqref{e:H0} hold. To prove such a result when $\gamma \in [-d,-2]$, we also need \emph{a priori} bounds on $\|f\|_{L^\infty}$ and on sufficiently high moments of $f$.

\begin{theorem}\label{t:gaussian}
Let $\gamma\in [-d,-2]$, and let $f$ be a bounded weak solution of the Landau equation \eqref{e:nondivergence} satisfying the hydrodynamic bounds \eqref{e:M0}, \eqref{e:E0}, and \eqref{e:H0}. Assume, in addition, that  
\[\int_{\R^d}|v|^p f(v)\dd v \leq P_0,\]
where $p$ is the smallest integer such that $p>\dfrac{d|\gamma|}{2+\gamma+d}$. Then there exists $\mu_0>0$ such that if 
\[f_{\rm in}(x,v) \leq C_0 e^{-\mu|v|^2},\]
for all $x\in \R^d,v\in\R^d$ and $\mu>0$, then 
\begin{equation}\label{e:Gaussian_decay}
	f(t,x,v) \lesssim e^{-\min\{\mu_0,\mu\}|v|^2},
\end{equation}
where $\mu_0$ and the implied constant in~\eqref{e:Gaussian_decay} depend on $C_0$, $M_0$, $E_0$, and $\|f\|_{L^\infty([0,T_0]\times \R^d\times \R^d)}$. If $\gamma \leq -d/2-1$, then the implied constant in~\eqref{e:Gaussian_decay} also depends on the time of existence $T_0$.    
\end{theorem}

\begin{proof}
First, assume that $\gamma \in (-d/2-1,-2]$. Fix $\mu_0>0$ to be determined and let $\overline \mu = \min\{\mu,\mu_0\}$.  Proceeding as in the proof of \cite[Theorem 1.2]{cameron2017landau}, we claim that $\phi(t,x,v) = e^{-\overline\mu|v|^2}$ is a supersolution to the linear Landau equation 
\begin{equation}\label{e:linear-landau}
\partial_t \phi + v\cdot \nabla_x \phi = \tr(\overline a D_v^2 \phi) + \overline c \phi,
\end{equation}
for $|v|$ large, where $\overline a$ and $\overline c$ are defined in terms of $f$. Since $\phi$ is radial in $v$, we have
\begin{equation*}
\partial_{v_i}\partial_{v_j}\phi = \frac {\partial_{rr}\phi}{|v|^2}   v_i v_j + \frac{\partial_r\phi}{|v|} \left( \delta_{ij} - \frac{v_iv_j}{|v|^2}\right) = \left[\frac {4\overline\mu^2|v|^2 - 2\overline\mu}{|v|^2}   v_i v_j -  2\overline\mu  \left( \delta_{ij} - \frac{v_iv_j}{|v|^2}\right)\right] e^{-\overline\mu |v|^2}.
\end{equation*}
Proposition \ref{p:a} and Lemma \ref{l:very-soft} from the appendix imply
\begin{align*}
\overline a_{ij}\partial_{v_i}\partial_{v_j}\phi &\leq \left[(4\overline\mu^2(1+|v|)^2 - 2\overline\mu)C_1 (1+|v|)^{\gamma}  -  2\overline\mu C_2 (1+|v|)^{\gamma+2}\right] e^{-\overline\mu |v|^2}\\
&= \left((4\overline\mu^2 C_1 - 2\overline\mu C_2) (1+|v|)^{\gamma+2} - 2\overline\mu C_1 (1+|v|)^{\gamma}\right) e^{-\overline\mu|v|^2}\\
&\leq -C (1+|v|)^{\gamma+2} \phi(v),
\end{align*}
for $|v|$ sufficiently large, provided that we choose $\mu_0 < C_2 / (2C_1)$, where we use the convention that repeated indices are summed over. With the bound on $\overline c$ from Lemma \ref{l:very-soft}, this implies
\[\overline a_{ij} \partial_{v_i}\partial_{v_j}\phi  + \overline c \phi \leq \left[-C(1+|v|)^{\gamma+2} + C(1+|v|)^{\gamma+2-\eps}\right]\phi(v).\]
The first term on the right-hand side dominates for large $|v|$, and we have 
\begin{equation}\label{e:aij-phi}
\overline a_{ij} \partial_{v_i}\partial_{v_j}\phi  + \overline c \phi \leq -C |v|^{\gamma+2} \phi
\end{equation} 
for $|v|\geq R_0$ for some large $R_0$. Choose $C_f$ such that $C_f\phi(t,x,v) > \|f\|_{L^\infty}$ for all $|v|\leq R_0$ and such that $C_f\phi(0,x,v) > f(0,x,v)$ for all $(x,v) \in \R^d \times \R^d$.  In the second inequality we used that $\overline\mu \leq \mu$.  Define the function
\[g(t,x,v) := [f(t,x,v) - C_f \phi(t,x,v)]_+.\]
If $|v|\leq R_0$, then $g(t,x,v) = 0$ by our choice of $C_1$. If $|v|> R_0$, then by \eqref{e:aij-phi}, $\phi$ is a supersolution to \eqref{e:linear-landau}. We conclude $g(t,x,v)$ is a subsolution of $\partial_t g + v\cdot \nabla_x g \leq \overline a_{ij}\partial_{v_iv_j} g + \overline c g$ in its entire domain; hence, by the maximum principle~\cite[Lemma A.2]{cameron2017landau}, we have $g\leq 0$ for all $t>0$, so $f(t,x,v) \leq C_1 \phi(t,x,v)$ for all $t>0$ for which $f$ is defined. 

If $\gamma \leq -d/2-1$, the above argument does not apply because we do not have enough \emph{a priori} decay in $\overline c$ to conclude \eqref{e:aij-phi}. For this case, we define $h(t,x,v) = f(t,x,v)e^{\mu|v|^2}$. From the equation \eqref{e:nondivergence} for $f$, we have
\begin{align*}
\partial_t h + v\cdot \nabla_x h &= e^{\mu|v|^2}\left( \tr\left[\overline a D_v^2(e^{-\mu|v|^2} h)\right] + \overline c e^{-\mu|v|^2} h\right)\\
&= \tr\left[\overline a D_v^2 h\right] - 4 \mu v\cdot (\overline a \nabla_v h) + \left(\overline c - 2\mu \, \tr (\overline a)  + 4 \mu^2 \overline a_{ij} v_i v_j\right) h.
\end{align*}
Lemma \ref{l:very-soft} implies that $\|\overline c -2\mu \, \tr (\overline a)  + 4 \mu^2 \overline a_{ij} v_i v_j\|_{L^\infty([0,T_0]\times \R^{2d})} \leq C_0$ for some $C_0$, so that $\tilde h(t,x,v) = e^{-C_0t}h(t,x,v)$ is a supersolution of $\partial_t \tilde h + v\cdot \nabla_x \tilde h = \tr(\overline a D_v^2 \tilde h) + \tilde b\cdot \nabla_v \tilde h$ with bounded drift $\tilde b_j = -4\mu v_i \overline a_{ij}$. The maximum principle for this class of equations (see for example \cite[Proposition A.1]{cameron2017landau}) implies $h(t,x,v) \leq e^{C_0t}f_{\rm in}(x,v)e^{\mu|v|^2}$, which is uniformly bounded on any finite time interval. Note that, since $\|f\|_{L^\infty([0,T_0]\times \R^d\times\R^d)}$ is finite, this argument also applies in the case $\gamma = -d$. 
\end{proof}


We are now in a position to prove our main result.

\begin{proof}[Proof of Theorem \textup{\ref{t:main}}]
Let $f$ be a weak solution of the Landau equation \eqref{e:divergence} such that $f_{\rm in}(x,v) = f(0,x,v) \lesssim e^{-\mu|v|^2}$ for some $\mu>0$. Without loss of generality, we may assume $\mu \leq \mu_0$, with $\mu_0$ as in the statement of the theorem.  By applying \cite[Theorem~1.2]{cameron2017landau} if $\gamma \in (-2,0)$ or \Cref{t:gaussian} if $\gamma \in [-d,-2]$, we see that, for all $(t,x,v) \in [0,T_0]\times\R^d \times \R^d$,
\begin{equation}\label{e:decay}
f(t,x,v) \lesssim e^{-\mu|v|^2},
\end{equation}
where the implied constant is independent of $T_0$ if $\gamma>-d/2 -1$. The dependence of the implied constant in~\eqref{e:decay} on $T_0$ in the case $\gamma \leq -d/2-1$ propagate to the rest of our estimates. Throughout this proof, as we absorb algebraic-in-$v$ factors into factors with Gaussian decay in $v$, $\mu'$ denotes a changing, positive constant, with $\mu'<\mu \leq \mu_0$.  The constant $\mu'$ changes only finitely many times, by an arbitrarily small amount, so the final conclusion is valid for any $\mu'<\mu$.

Let $f_{z_0}$ be as in Lemma \ref{l:T} with base point $z_0\in [0,T_0]\times \R^d\times \R^d$.  Since \Cref{l:T} locally controls the coefficients in the equation for $f_{z_0}$~\eqref{e:isotropic}, we may apply \cite[Theorem~2]{golse2016} to obtain:
\begin{align*}
|f_{z_0}|_{\alpha,Q_{1/2}} \lesssim \|f_{z_0}\|_{L^2(Q_1)}+|\overline C f_{z_0}|_{0,Q_1},
\end{align*}
for some $\alpha\in (0,1)$. Using the Gaussian decay of $f$~\eqref{e:decay}, this implies $|f_{z_0}|_{\alpha,Q_{1/2}} \lesssim e^{-\mu'|v_0|^2}$. By rescaling, we have $|f_{z_0}|_{\alpha,Q_{1}} \lesssim e^{-\mu'|v_0|^2}$. Next, Lemma~\ref{l:coefficients} with $M=p=0$, along with the local upper bounds on $\overline A$ and $\overline C$ in Lemma \ref{l:T}, implies that the coefficients $\overline A$ and $\overline C$ in \eqref{e:isotropic-nondivergence} satisfy
\[ \left|\overline A\right|_{2\alpha/3,Q_{1}} + \left|\overline C\right|_{2\alpha/3,Q_{1}} \lesssim (1+|v_0|)^{k_0},\]
for some $k_0\in \R$, with $\alpha$ as above. 
We apply the Schauder estimate\footnote{Technically, \Cref{t:weak-schauder} does not apply to $f$ since it is not sufficiently regular; however, a standard mollification argument allows us to sidestep this potential issue.  We omit the details.}, Theorem \ref{t:weak-schauder}(a), to $f_{z_0}$ in $Q_1$ with $\alpha' = 2\alpha/3$ to obtain
\[ [f_{z_0}]_{1+\alpha'/3,Q_{1/2}} \leq C([\overline C f_{z_0}]_{\alpha',Q_1} + |\overline A|_{\alpha',Q_1}^{p}|f_{z_0}|_{0,Q_1}) \lesssim  e^{-\mu'|v_0|^2},\]
for any $z_0\in (0,T_0]\times \R^d\times \R^d$, where $p>0$ depends on $\alpha$. By Lemma \ref{l:coefficients} again, this implies $\overline A, \overline C\in C^{1+\alpha''}(Q_{1/2})$, with $\alpha'' = 2\alpha'/3$ and 
\begin{align*}
 \left|\overline A\right|_{1+\alpha'',Q_{1/2}} &\lesssim r_1^{-2} (1+|v_0|)^{k_1} \lesssim (1+t_0^{-1})(1+|v_0|)^{k_1},\\
  \left|\overline C\right|_{1+\alpha'',Q_{1/2}}  &\lesssim  (1+|v_0|)^{\ell_1}, 
\end{align*}
for $k_1, \ell_1 \in \R$. We can now apply Theorem \ref{t:weak-schauder}(b) to obtain
\begin{align*}
 [\partial_t f_{z_0}]_{\alpha'',Q_{1/4}}& + [\nabla_x f_{z_0}]_{\alpha'',Q_{1/4}} + [D_v^3 f_{z_0}]_{\alpha'',Q_{1/4}}\\
 &\lesssim (|\overline C f_{z_0}|_{1+\alpha'',Q_{1/2}} + |\overline A|_{1+\alpha'',Q_{1/2}}^q|f_{z_0}|_{0,Q_{1/2}})\\
 &\lesssim (1+t_0^{-q}) e^{-\mu'|v_0|^2}.
 \end{align*}
where $q>0$ depends on $\alpha$. Again, by taking a larger constant we have
\[[D_v^3 f_{z_0}]_{\alpha'',Q_{1}} +  [\partial_t f_{z_0}]_{\alpha'',Q_{1}} + [\nabla_x f_{z_0}]_{\alpha'',Q_{1}} \lesssim (1+t_0^{-q}) e^{-\mu'|v_0|^2}. \]
From here, we can inductively apply Theorem \ref{t:weak-schauder}(a) and (b) to conclude $f_{z_0}\in C^{\infty}(Q_1)$. In more detail, assume that all partial derivatives $\partial_t^j \partial_x^\beta\partial_v^\eta f_{z_0}$ with
\begin{equation}\label{e:M}
 2j + 3|\beta| + |\eta| \leq M 
 \end{equation}
 exist in $C^{\alpha}(Q_1)$ for some $\alpha>0$, and that for every such partial derivative $\partial f_{z_0}$ and $z_0\in (0,T_0]\times \R^d\times \R^d$, we have
 \begin{equation}\label{e:partial}
  [\partial f_{z_0}]_{\alpha,Q_{1}} \leq C\left(1+t_0^{-q}\right) e^{-\mu'|v_0|^2},
 \end{equation}
for some $q>0$. Then Lemma \ref{l:coefficients} implies that $\overline A$ and $\overline C$ in \eqref{e:isotropic-nondivergence} satisfy 
\begin{equation}\label{e:AC}
\begin{split} 
\left[\partial\overline A\right]_{1+\alpha',Q_{1/2}} &
	\lesssim (1+t_0^{-q'})(1+|v_0|)^k\\
  \left[\partial\overline C\right]_{1+\alpha',Q_{1/2}} &
  	\lesssim (1+t_0^{-q'+1}) (1+|v_0|)^\ell,
\end{split}
\end{equation}
for some $q'>0$ and $k,\ell\in \R$. Letting $\partial = \partial_t^j\partial_x^\beta\partial_v^\eta$ be a partial derivative satisfying \eqref{e:M}, we can therefore differentiate equation \eqref{e:isotropic-nondivergence} to obtain an equation for $\partial f_{z_0}$ of the form
\begin{align*}
\partial_t (\partial f_{z_0}) + v\cdot \nabla_x (\partial f_{z_0})  &= \tr(\overline A(z)\partial f_{z_0}) + \overline C(z)\partial f_{z_0} + \mathcal F(f_{z_0}(z),\overline A(z),\overline C(z)),
\end{align*} 
for some differential operator $\mathcal F$ of order at most $M$ (counted with the scaling of \eqref{e:M}). Applying Theorem \ref{t:weak-schauder}(a) and our inductive hypothesis \eqref{e:partial}, we have
\begin{align*} 
[\partial f_{z_0}]_{1+\alpha',Q_{1/2}} &\lesssim \left([\overline C(z) \partial f_{z_0}+ \mathcal F(f_{z_0}(z),\overline A(z),\overline C(z))]_{\alpha',Q_1} + |A|_{\alpha',Q_1}^p|f_{z_0}|_{0,Q_1}\right)\\
 &\lesssim \left(1+t_0^{-q''}\right)e^{-\mu'|v_0|^2},
\end{align*}
with $q''>0$. By \eqref{e:AC}, we have enough regularity of $\overline C(z)$ and $\mathcal F(f_{z_0}(z),\overline A(z),\overline C(z))$ to apply Theorem \ref{t:weak-schauder}(b):
\[[D_v^3 \partial f_{z_0}]_{\alpha'',Q_{1/4}} +  [\partial_t \partial f_{z_0}]_{\alpha'',Q_{1/4}} + [\nabla_x \partial f_{z_0}]_{\alpha'',Q_{1/4}} \lesssim  \left(1+t_0^{-q'''}\right) e^{-\mu'|v_0|^2}. \]
As above, we may replace $Q_{1/4}$ with $Q_1$ by taking a larger implied constant. Such an estimate holds for each partial derivative $\partial f_{z_0}$ satisfying \eqref{e:M}, so we have shown \eqref{e:partial} holds with some $q>0$ for $\partial_t^j\partial_x^\beta\partial_v^\eta f_{z_0}$ whenever
\[ 2j + 3|\beta| + |\eta| \leq M + 3.\]
We conclude $f_{z_0}\in C^\infty(Q_1)$ for any $z_0\in (0,T_0]\times \R^d\times \R^d$. By Proposition \ref{p:f0-f}, we have that $f\in C^\infty((0,T_0]\times \R^d\times \R^d))$ with the pointwise estimates \eqref{e:pointwise}.
\end{proof}

\appendix

\section{Bounds on the coefficients of the Landau equation}\label{s:A}

In this appendix, we collect the available bounds on the coefficients $\overline a$, $\overline b$, and $\overline c$ in the Landau equation \eqref{e:divergence} with soft potentials ($\gamma\in [-d,0)$). The estimates in Propositions \ref{p:a} and \ref{p:bc} were derived in \cite{silvestre2015landau} and \cite{cameron2017landau}. Earlier, corresponding bounds in the case $\gamma \geq 0$ were shown in \cite{desvillettes2000landau}.
\begin{proposition}\label{p:a}
Let $f:[0,T_0]\times \R^d\times\R^d\to \R_+$ satisfy the bounds \eqref{e:M0}, \eqref{e:E0}, and \eqref{e:H0}, and let $\overline a$ be defined by \eqref{e:a}. If $\gamma\in [-d,0)$, then for unit vectors $e\in \R^d$,
\begin{equation}\label{e:aij-lower}
 \overline a_{ij}(t,x,v) e_i e_j \geq c\begin{cases} (1+|v|)^{\gamma}, & e\in \mathbb S^{d-1}, \\
(1+|v|)^{\gamma+2},& e\cdot v = 0.\end{cases}
\end{equation}
If $\gamma\in [-2,0)$, then $\overline a$ satisfies the upper bound
\begin{equation}\label{e:aij-upper}
 \overline a_{ij}(t,x,v) e_i e_j \leq C\begin{cases} (1+|v|)^{\gamma+2}, &e\in \mathbb S^{d-1},\\
(1+|v|)^{\gamma},  & e\cdot v = |v|,\end{cases}
\end{equation}
and if $\gamma\in [-d,-2)$, 
\begin{equation}
\overline a_{ij}(t,x,v)e_i e_j \leq C\|f(t,x,\cdot)\|_{L^\infty(\R^d)}^{-(\gamma+2)/d}, \quad e\in \mathbb S^{d-1}.
\end{equation}
The constants $c$ and $C$ depend on $d$, $\gamma$, $m_0$, $M_0$, $E_0$, and $H_0$.
\end{proposition}

\begin{proposition}\label{p:bc}
Let $f$ be as in Proposition \textup{\ref{p:a}}. The coefficients $\overline b$ and $\overline c$ defined by \eqref{e:b} and \eqref{e:c} respectively, satisfy the upper bounds
\begin{equation}\label{e:b-upper}
\left|\overline b(t,x,v)\right| \leq C\begin{cases} (1+ |v|)^{\gamma+1}, &-1 \leq \gamma <0,\\[2ex]
(1+|v|)^{\gamma+1}(1+\|f\|_{L^\infty(B_{1}(v))})^{-(\gamma+1)/d}, &\dfrac{-3d-2}{d+2}\leq \gamma < -1,\\[2ex]
(1+|v|)^{-2-2(\gamma+1)/d}\left(1+\|f\|_{L^\infty(B_{1}(v))}  \right)^{-\gamma/d}, &-d \leq \gamma < \dfrac{-3d-2}{d+2},\end{cases}
\end{equation}
and 
\begin{equation}
\overline c(t,x,v) \leq C\begin{cases} (1+|v|)^\gamma(1+\|f\|_{L^\infty(B_{1}(v))})^{-\gamma/d}, &\dfrac{-2d}{d+2}\leq \gamma < 0,\\
(1+|v|)^{-2-2\gamma/d}\left(1+\|f\|_{L^\infty(B_{1}(v))}  \right)^{-\gamma/d}, &-d < \gamma < \dfrac{-2d}{d+2},\end{cases}
\end{equation}
where the constants depend on $d$, $\gamma$, $M_0$, and $E_0$.
\end{proposition}
Finally, we show that when $\gamma \in [-d,-2]$, the coefficients $\overline a$ and $\overline c$ still have the appropriate decay to prove Theorem \ref{t:gaussian}, if sufficiently many moments of $f$ are finite.
\begin{lemma}\label{l:very-soft}
Let $\gamma\in [-d,-2]$, and let $f:[0,T_0]\times \R^d\times\R^d\to\R$ be a bounded function satisfying \eqref{e:M0} and \eqref{e:E0}. Assume in addition that 
\[\int_{\R^d}|v|^p f(v)\dd v \leq P_0,\]
where $p$ is the smallest integer such that $p>\dfrac{d|\gamma|}{2+\gamma+d}$. Then the upper bounds \eqref{e:aij-upper} hold, 
with constants depending on $d$, $\gamma$, $M_0$, $E_0$, $P_0$, and $\|f\|_{L^\infty([0,T_0]\times \R^{d}\times\R^d)}$. 

If, in addition, $\gamma > -d/2-1$, there is an $\eps>0$ depending on $d$ and $\gamma$ such that 
\[\overline c(t,x,v) \leq C(1+|v|)^{\gamma+2-\eps},\]
with $C$ depending on the same quantities.
\end{lemma}
\begin{proof}
For any $e\in \mathbb S^{d-1}$, the formula \eqref{e:a} implies
\begin{align*}
\overline a_{ij}(t,x,v)e_i e_j &= a_{d,\gamma} \int_{\R^d} \left(1 - \left(\frac {w\cdot e}{|w|}\right)^2\right)|w|^{\gamma+2} f(v-w)\dd w \\
&\lesssim \int_{\R^d} |w|^{\gamma+2} f(v-w)\dd w.
\end{align*}
Let $r := \frac 1 2 |v|^{(\gamma+2)/(\gamma+2+d)}$, $R = |v|/2$, and define
\begin{align*}
I_1 = \int_{B_{r}} |w|^{\gamma+2}  &  f(v-w)\dd w, \quad I_2 = \int_{B_{R}\setminus B_{r} } |w|^{\gamma+2} f(v-w)\dd w,
\\& \quad I_3 = \int_{\R^d \setminus B_{R}} |w|^{\gamma+2} f(v-w)\dd w.
\end{align*}
We have
\[I_1 \lesssim \|f\|_{L^\infty} r^{d+\gamma+2} \lesssim |v|^{\gamma+2},\]
\begin{align*}
I_2 \lesssim r^{\gamma+2} |v|^{-p}\int_{B_{R}} |v-w|^p f(v-w) \dd w &\lesssim P_0 |v|^{-p+(\gamma+2)^2/(d+\gamma+2)}.
\end{align*}
Our choice of $p$ implies $-p< d(\gamma+2)/(d+\gamma+2)$, so that $I_2 \lesssim |v|^{\gamma+2}$. Finally, for $|w|\geq |v|/2$, we have $|w|^{\gamma+2} \lesssim  |v|^{\gamma+2}$, and
\begin{align*}
I_3 \lesssim |v|^{\gamma+2} \int_{\R^d\setminus B_R}f(v-w)\dd w \leq M_0|v|^{\gamma+2}.
 \end{align*}

If $e$ is parallel to $v$, then proceeding as in \cite[Lemma 2.1]{cameron2017landau}, 
 we have
\begin{align*}
 \int_{\R^d} \left(1 - \left(\frac {w\cdot e}{|w|}\right)^2\right)|w|^{\gamma+2} f(v-w)\dd w & = \int_{\R^d} \left(1 - \left(\frac {(v-z)\cdot e}{|v-z|}\right)^2\right)|v-z|^{\gamma+2} f(z)\dd z\\
&=  \int_{\R^d} \left(|v-z|^2 - \left(|v|-z\cdot e\right)^2\right)|v-z|^{\gamma} f(z)\dd z\\
&=  \int_{\R^d} \left(|z|^2 - (z\cdot e)^2\right)|v-z|^{\gamma} f(z)\dd z\\
&= \int_{\R^d} |z|^2 \sin^2\theta|v-z|^{\gamma} f(z)\dd z,
\end{align*}
where $\theta$ is the angle between $v$ and $z$. We may assume $|v|>2$. Let $R = |v|/2$ and $q = \dfrac{p(p-2)}{p+d}$. By our choice of $p$, we have $(\gamma+2)p/q > -d$. If $z\in B_R(v)$, then $|\sin \theta| \leq |v-z|/|v|$, $|z|\lesssim |v|$, and
\begin{align*}
\int_{B_R(v)} |z|^2\sin^2\theta|v-z|^\gamma f(z)\dd z &\leq |v|^{-2}\int_{B_R(v)} |z|^2 |v-z|^{\gamma+2}f(z)\dd z\\
&\leq |v|^{-p+q}\|f\|_{L^\infty}^{q/p}\int_{B_R(v)} |z|^{p-q} f(z)^{(p-q)/p}|v-z|^{\gamma+2} \dd z\\
&\lesssim |v|^{-p+q}\left(\int_{B_R(v)}|z|^pf(z)\dd z\right)^{(p-q)/p}\left(\int_{B_R(v)} |v-z|^{(\gamma+2)p/q}\dd z\right)^{q/p}\\
&\lesssim |v|^{-p+q} E_0^{(p-q)/p} \left(|v|^{(\gamma+2)p/q + d}\right)^{q/p} \lesssim |v|^{\gamma}.
\end{align*}
If $|v-z| \geq R=|v|/2$, then $|v-z|^\gamma \lesssim |v|^\gamma$, and we have
\begin{align*}
\int_{\R^d\setminus B_R(v)} |z|^2 \sin^2\theta|v-z|^\gamma f(z)\dd z &\lesssim  |v|^{\gamma} \int_{\R^d\setminus B_R(v)} |z|^2 f(z)\dd z \lesssim E_0 |v|^{\gamma}.
\end{align*}

For $\overline c$, our choice of $p$ and the restriction that $\gamma > -d/2 -1$ implies there is an $\eps>0$ such that $-p + \dfrac{\gamma(\gamma+2-\eps)}{d+\gamma} < \gamma + 2$. Define $r = \frac 1 2 |v|^{(\gamma+2-\eps)/(d+\gamma)}$, $R = |v|/2$, and $I_1$, $I_2$, $I_3$ as above. The same method implies $I_1 + I_2 + I_3 \lesssim |v|^{\gamma+2-\eps}$.
\end{proof}

\section*{Ethical Statement}

\noindent {\bf Funding:} Both authors were partially supported by National Science Foundation grant DMS-1246999. CH was partially supported by NSF grant DMS-1907853. SS was partially supported by a Ralph E. Powe Award from ORAU.

\noindent {\bf Conflict of interest:} The authors declare that they have no conflicts of interest.

\bibliographystyle{abbrv}


\end{document}